\documentclass[numbers,webpdf,imaiai]{ima-authoring-template}%

\theoremstyle{thmstyletwo}%
\newtheorem{remark}{Remark}%

\numberwithin{equation}{section}

\newtheorem{lemma}{Lemma}[section]
\newtheorem{prop}{Proposition}[section]
\newtheorem{problem}{Problem}

\RequirePackage{amssymb,amsmath,amsthm,bm,mathtools,amsfonts}
\RequirePackage{graphicx,xcolor,subfig}
\usepackage{epstopdf}
\graphicspath{{./fig/}}

\usepackage{relsize}
\usepackage{fixmath} 
\usepackage{scalerel}

\newcommand{\norm}[1]{\left\Vert {#1} \right\Vert}

\newcommand{\defeq}{\vcentcolon=}

\DeclareMathOperator{\diver}{{\rm div}}
\DeclareMathOperator{\curl}{{\rm curl}}
\DeclareMathOperator{\rot}{{\rm curl}}

\newcommand{\xx}{\boldsymbol{x}}

\newcommand{\uu}{\boldsymbol{u}}
\newcommand{\vv}{\boldsymbol{v}}

\newcommand{\ww}{\boldsymbol{w}}

\newcommand{\nn}{\boldsymbol{n}}
\newcommand{\VV}{\boldsymbol{V}}

\newcommand{\R}{\mathbb{R}}
\newcommand{\Pk}{\mathbb{P}}

\newcommand{\Mk}{\mathcal{M}}

\newcommand{\cfan}[1]{\textnormal{\textbf{\detokenize{#1}}}}

\begin{document}

\copyrightyear{2021}
\vol{00}
\pubyear{2021}
\access{Advance Access Publication Date: Day Month Year}
\appnotes{Paper}
\copyrightstatement{Published by Oxford University Press on behalf of the Institute of Mathematics and its Applications. All rights reserved.}
\firstpage{1}


\title[MVEM for linear acoustic wave equation]{Mixed Virtual Element approximation of linear acoustic wave equation}

\author{Franco Dassi\ORCID{0000-0001-5590-3651}
\address{\orgdiv{Dipartimento di Matematica e Applicazioni}, \orgname{Universit\`a degli Studi di Milano Bicocca}, \orgaddress{\street{Via R. Cozzi, 55}, \postcode{20125}, \state{Milano}, \country{Italy}}}}
\author{Alessio Fumagalli\ORCID{0000-0003-3039-5309}
\address{\orgdiv{MOX - Dipartimento di Matematica}, \orgname{Politecnico di Milano}, \orgaddress{\street{P.za L. da Vinci, 32}, \postcode{20133}, \state{Milano}, \country{Italy}}}}
\author{Ilario Mazzieri*\ORCID{0000-0003-4121-8092}
\address{\orgdiv{MOX - Dipartimento di Matematica}, \orgname{Politecnico di Milano}, \orgaddress{\street{P.za L. da Vinci, 32}, \postcode{20133}, \state{Milano}, \country{Italy}}}}
\author{Giuseppe Vacca\ORCID{0000-0002-9035-5731}
\address{\orgdiv{Dipartimento di Matematica}, \orgname{Universit\`a degli Studi di Bari}, \orgaddress{\street{via E. Orabona, 4}, \postcode{70125}, \state{Bari}, \country{Italy}}}}

\authormark{Franco Dassi et al.}

\corresp[*]{Corresponding author: \href{email:ilario.mazzieri@polimi.it}{ilario.mazzieri@polimi.it}}

\received{Date}{0}{Year}
\revised{Date}{0}{Year}
\accepted{Date}{0}{Year}


\abstract{We design a Mixed Virtual Element Method for the approximated solution to the first-order form of the acoustic wave equation. In absence of external load, the semi-discrete method exactly conserves the system energy. To integrate in time the semi-discrete problem we consider a classical $\theta$-method scheme.
We carry out the  stability and convergence analysis in the energy norm for the semi-discrete problem  showing optimal rate of convergence with respect to the mesh size.
We further study the property of energy conservation for the fully-discrete system.
Finally, we present some verification tests as well as  engineering application of the method.}
\keywords{Mixed Virtual Elements; acoustics wave equations; polygonal meshes; energy conservation.
}


\maketitle
\section{Introduction}

The numerical simulation of acoustic, elastic or electromagnetic wave propagation finds application in many scientific disciplines, including aerospace, geophysics, civil engineering,
telecommunications, and medicine for instance.

The present work considers a Mixed Virtual Element method on general polytopal grids for the discretization of the acoustic wave equation written as a first order system of hyperbolic partial differential equations.

In general, mixed methods consider the discretization of vector fields in some $H(\text{div})$-conforming spaces while scalar fields in some $L^2$ spaces. Classes of mixed methods include the well known Raviart-Thomas (RT)~\cite{Raviart1977,Roberts1991,Arnold2005}, the Brezzi-Douglas-Marini (BDM)~\cite{Brezzi1985,Nedelec1986,Brezzi1987,Boffi2013} finite element schemes and more recently the
Mixed Virtual Element Methods (MVEM)~\cite{BeiraodaVeiga2014b,dassi2021}. The VEM, introduced firstly in \cite{BeiraodaVeiga2013a}, have been applied recently to various differential problems, including elasticity \cite{ARTIOLI2017}, Stokes \cite{BLV17,CMM21}, Navier-Stokes \cite{BLV18}, Cahn-Hilliard
\cite{ABSV16}, Darcy \cite{dassi2021}, Helmholtz \cite{MPP19}, Maxwell \cite{BEIRAODAVEIGA2021} and wave \cite{VACCA2017882,DassiFumagalliMazzieriScottiVacca2021,AntoniettiManziniMazzieriMouradVerani_2019} equations.

The major benefit of using VEM, instead of classical approaches, is the fact that it gives the opportunity to preserve at the discrete
level some important properties valid at the continuous level. In particular, it is possible to design discrete spaces with global high regularity, which preserve the polynomial divergence/curl,  that are robust with respect to mesh-locking phenomena.
Moreover, VEM can handle general polytopal meshes, that are particularly useful to account for small features in the model (such as cracks, holes and inclusions), and treat in an automatic way hanging nodes, movable meshes and adaptivity.

The very first analysis of RT finite element discretization applied for  the spatial approximation of the acoustic wave equation is presented by Geveci in \cite{Geveci1988}. He showed that
even if the RT finite elements conserve the energy of the system, when a time discretization is applied, the fully-discrete method produces an implicit time-marching scheme which is inefficient, since the mass matrix is nondiagonal. For this reason it is preferred to use  mass lumping techniques, see e.g.  \cite{Becache2000,Cohen2001} or symplectic schemes that conserve a positive-definite perturbed energy functional \cite{Kirby2015}.
A slightly different formulation, in which two time derivatives appear on the vector variable and none on the other equation has been analysed in literature, see, for instance, \cite{Cowsar1990,Jenkins2002,Jenkins2007,Chung2006,HE202160}.

Discontinuous Galerkin (dG) methods have been also considered for the solution of the wave equation in the mixed form, see for instance \cite{Monk2005,Chung2006,Chung2009,COCKBURN2014,Moiola2018}. However, they have been mostly studied for the second order hyperbolic version: we mention \cite{riviere2003discontinuous,Grote2006} for the scalar case, \cite{riviere2007discontinuous,de2008interior,de2010stability,antonietti2012non,mercerat2015nodal,Delcourte2015,AntoniettiMarcatiMazzieriQuarteroni_2016} for the vectorial case and \cite{AntoniettiMazzieri_2017,AntoniettiBonaldiMazzieri_2019a,AntoniettiMazzieriMuhrNikolicWohlmuth_2019} where a dG approximation on polygonal grids is considered. Within the framework of polytopal methods we mention a recent work by \cite{Burman2022} where the Hybrid High-Order (HHO) method is applied to the wave equation in either first and second order form
and \cite{BdVLV} where the Mimetic Finite Differences is applied in the context of Hamiltonian wave equations.

Here, for the first time, a MVEM is considered for the solution of linear wave acoustics written as a system of first order partial differential equations. The analysis is carried out by taking inspiration from the approaches proposed  in~\cite{BeiraodaVeiga2016,BeiraodaVeiga2019} and \cite{dassi2021}.  Concerning the choice of the degrees of freedom, the proposed scheme can be seen as the extension on polytopal grids of RT finite elements. The integration in time of the semi-discrte problem is achieved by considering a $\theta$-method scheme.

The paper is organised as follows: in Section~\ref{sec:mathematical_model} we review the mathematical model, its weak formulation, a stability result and the energy conservation of the system.
In Section~\ref{sec:VEM} after introducing the virtual element spaces with the associated set of degrees of freedom and defining the discrete bilinear forms, we present the semi-discrete virtual element formulation and establish the well-posedness of the semi-discrete problem, the stability bound for
discrete solution in a discrete energy norm and the energy conservation.
In Section~\ref{sec:theory} we analyse the theoretical properties of the proposed method: by considering the Fortin operator introduced in \cite{dassi2021}, we recall the interpolation estimates. Then, we prove optimal order of convergence for the proposed method. Moreover we estimate the error between the energy of the exact solution and the energy of the virtual element solution.
In Section~\ref{sec:time_integration} we introduce the family of $\theta$-method schemes for integrating in time the semidiscrete problem in Section~\ref{sub:vem problem} and discuss the property of energy conservation.
In Section~\ref{sec:numExe} we provide some experiments to give numerical evidence of the behaviour
of the proposed scheme. Finally, Section \ref{sec:conclusion} is devoted to conclusions and future perspectives.

\subsection*{Notations and Preliminaries}

Throughout the paper we will follow the usual notation for Sobolev spaces
and norms as in \cite{Adams:1975}.
Let $\Omega \subset \R^2$ be the computational domain  with Lipschitz continuous boundary $\partial \Omega$ and external unit normal $\nn$,
we denote with $\xx = (x_1, \, x_2)$ the independent variable.
With a usual notation
the symbols $\nabla$ and $\rot$ denote the gradient and curl for scalar functions, while
$\diver$ denotes the divergence operator for vector fields.
For an open bounded domain $\omega$,
the norm in the space  $L^p(\omega)$ is denoted by $\|{\cdot}\|_{L^p(\omega)}$,
norm and seminorm in $H^{s}(\omega)$ are denoted respectively by
$\|{\cdot}\|_{s,\omega}$ and $|{\cdot}|_{s,\omega}$,
while $(\cdot,\cdot)_{\omega}$ and $\|\cdot\|_{0, \omega}$ denote the $L^2$-inner product and the $L^2$-norm (the subscript $\omega$ may be omitted when $\omega$ is the whole computational
domain $\Omega$).

We recall the following well known functional spaces which will be useful in the sequel
\begin{gather*}
    H(\diver, \omega) \defeq \{\bm{v} \in [L^2(\omega)]^2: \,   \diver
    \bm{v} \in L^2(\omega)\} \,,
    \\
    H(\rot, \omega) \defeq \{\bm{v} \in [L^2(\omega)]^2 :\,
    \rot \bm{v} \in L^2(\omega)\} \,,
\end{gather*}
and introduce the following spaces
\[
\VV \defeq \left\{ \bm{v} \in H(\diver, \Omega)
    \quad \text{s.t.} \quad \bm{v} \cdot \bm{n} = 0 \text{ on } \Gamma_N
    \right\} \,,
\quad
Q \defeq L^2(\Omega) \,,
\]
where $\Gamma_N \subset \partial \Omega$,  equipped with natural inner products and induced norms.

Since we are dealing with a time dependent problem, we will also consider the following Bochner spaces.
Let $T>0$, for space-time functions $v(\xx, t)$ defined on $\omega \times (0, T)$, we denote with $v_t$ the derivative with respect to the time variable.
Furthermore, using standard notations \cite{quarteroni-valli:book}, for a Banach space $V$ with norm $\|\cdot\|_V$, we introduce the space
\[
L^2(0, T; V)
\defeq \left\{ v \colon (0,T) \to V \,\,\,
\text{s.t.} \,\,\,
\text{$v$ measurable,} \,\,
\int_{0}^{T} \Vert v(t) \Vert_V^2  dt < +\infty
    \right\}.
\]
In similar way, for $n\geq 0$, we consider the space $C^n(0, T; V)$.

\section{Mathematical Model}\label{sec:mathematical_model}

Let $\Omega\subset\mathbb{R}^2$ be the polygonal domain.
The boundary  $\partial \Omega$ is divided in three parts, with mutually disjoint interiors, denoted  by $\Gamma_D$, $\Gamma_N$ and $\Gamma_R$, corresponding to Dirichlet, Neumann and Robin boundary conditions, respectively; one or two of them may be empty.

In a time interval $(0,T]$, for a piece-wise constant positive real valued function $c$ (representing the characteristic velocity of the medium), and a scalar source $f$, the following problem is set in $\Omega$.
\begin{problem}[Model problem]\label{pb:wave_model}
    Find $(\bm{u}, p)$ such that
    \begin{subequations}
    \begin{align*}
        \left \{
        \begin{aligned}
             \bm{u}_t(\bm x,t) - \nabla p(\bm x,t) &= \bm{0} \\
            c^{-2} p_t(\bm x,t) - \diver \bm{u}(\bm x,t) &= f(\bm x,t)
        \end{aligned}
        \right .
        \qquad \text{in } \Omega\times(0,T],
    \end{align*}
    supplied with the following boundary conditions
    \begin{equation*}%
    \left \{
    \begin{aligned}
        p(\bm x,t) &= g_D(\bm x,t) & \qquad \text{on } \Gamma_D\times (0,T],\\
        \bm{u}(\bm x,t) \cdot \bm{n} &= g_N(\bm x,t) & \qquad \text{on } \Gamma_N  \times (0,T], \\
        \bm{u}(\bm x,t) \cdot \bm{n} + \alpha^{-1} c^{-1}p(\bm x,t) &= g_R(\bm x,t) & \qquad \text{on } \Gamma_R  \times (0,T],
    \end{aligned}
    \right.
    \end{equation*}
    being $\alpha >0 $ an impedance parameter, and $g_D,g_N$ and $g_R$ given functions, and initial conditions
    \begin{equation*}
    \left \{
    \begin{aligned}
        p(\bm x,0) &= p_0(\bm x) & \qquad \text{in } \Omega,\\
        \bm{u}(\bm x,0)  &= \bm u_0(\bm x) & \qquad \text{in } \Omega.
    \end{aligned}
    \right.
    \end{equation*}
    \end{subequations}
\end{problem}
This problem describes the space-time variation of particle velocity $\bm u$ and acoustic pressure $p$ in a heterogeneous medium where waves propagate with  characteristic velocity $c=\sqrt{\frac{\mu}{\rho}}$, being $\rho>0$ the density  and $\mu>0$ the viscosity of the medium, respectively.
%
Notice that the condition on $\Gamma_R \times [0, T ]$ is known in the literature as impedance boundary condition and includes the low-order absorbing condition when $\alpha = 1$ and $g_R=0$, \cite{Moiola2018}.
In the following we suppose $\bm u_0 \in Q^2$, $p_0 \in Q$,
$f \in L^{2}(0, T; Q)$ and for the sake of presentation we set $g_D = g_N = g_R =0$ and $\alpha=1$. The general case, i.e., with non-homogeneous boundary conditions con be treated similarly.
Next, we define the following bilienar forms
\begin{equation}
\label{eq:continuous_forms}
    \begin{aligned}
        m(\cdot, \cdot)&\colon \VV \times \VV \to \mathbb{R} \qquad
        &m(\bm{u}, \bm{v}) &\defeq (\bm{u}, \bm{v})_\Omega
        \quad &&\forall (\bm{u},
        \bm{v}) \in \VV \times \VV,\\
        n(\cdot, \cdot)&\colon Q \times Q \to \mathbb{R} \qquad
        &n(p, q) &\defeq (c^{-2} p, q)_\Omega \quad &&\forall (p,
        q) \in Q \times Q,\\
        b(\cdot, \cdot)&\colon \VV \times Q \to \mathbb{R} \qquad
        &b(\bm{u}, q) &\defeq (\diver \bm{u}, q)_\Omega \quad &&\forall(\bm{u}, q)\in \VV
        \times Q,\\
        r(\cdot, \cdot)&\colon \VV \times \VV \to \mathbb{R} \qquad
        &r(\bm u, \bm v) &\defeq (c \bm u \cdot \bm n, \bm v \cdot \bm n)_{\Gamma_R} \quad &&\forall (\bm u,
        \bm v) \in \VV \times \VV,
    \end{aligned}
\end{equation}
and the linear functional associated to given data is defined as
\begin{align*}
    \begin{aligned}
    &F(\cdot) \colon Q\to \mathbb{R} \quad
    &&F(q) \defeq (f, q)_\Omega \quad &\forall q \in Q.
    \end{aligned}
\end{align*}
Then the weak formulation of Problem \ref{pb:wave_model} reads as follows
\begin{problem}[Weak problem]\label{pb:wave_weak}
Find the velocity $\uu \in L^2(0,T; \VV) \cap C^0(0,T; Q^2)$,
and the pressure $p \in L^2(0,T; Q) \cap C^0(0,T; Q)$
s.t.
    \begin{align*}
       \left \{
        \begin{aligned}
            m(\bm u_t, \bm{v}) + b(\bm{v}, p) + r(\bm u, \bm{v}) &= 0 && \text{$\forall \bm{v} \in \VV$},
            \\
            n(p_t, q) - b(\bm{u}, q) &= F(q) &&
            \text{$\forall q \in Q$},
        \end{aligned}
         \right.
    \end{align*}
with initial condition $\bm u(\cdot, 0) = \bm u_0$ and $p(\cdot,0)=p_0$ in $\Omega$.
\end{problem}

Using standard arguments is possible to prove that Problem  \ref{pb:wave_weak} is well posed, \cite{BrezziFortin1991,Jenkins2002,Moiola2018} and satisfies the following stability estimate, cf. also \cite{Egger2018}
\begin{equation*}
\sup_{0\leq t \leq T} \|(\bm u, p)(t)\|_\mathcal{E}   \leq  \int_{0}^T \norm{c f(s)}_{0,\Omega} \, {\rm d}s + \|(\bm u_0, p_0)\|_\mathcal{E},
\end{equation*}
where
\begin{equation}\label{def:energy_norm}
\|(\bm v, q)\|^2_\mathcal{E} := \| \bm v\|^2_{0,\Omega} + \| c^{-1} q \|^2_{0,\Omega} \quad \forall \; (\bm v, q) \in [L^2(\Omega)]^2 \times Q \,.
\end{equation}

We finally observe that if $f = 0$ and $\Gamma_R = \emptyset$ Problem \ref{pb:wave_weak} is energy conservative, i.e. the solution $(\uu, p)$ of Problem \ref{pb:wave_weak} satisfies
\begin{equation}
\label{eq:cons-cont}
\|(\bm u, p)(t)\|_\mathcal{E} =
\|(\bm u_0, p_0)\|_\mathcal{E}
\qquad
\forall t \in [0,T] \,.
\end{equation}

\section{Mixed Virtual Elements}
\label{sec:VEM}
In this Section~we describe the virtual element discretization of Problem \ref{pb:wave_weak} on general polygonal meshes.
In particular, in Section~\ref{sub:mesh} we introduce the assumptions on the regularity of the polygonal meshes together with the definition of crucial projector operators that will be fundamental in the construction of the VE discretization.
In Section~\ref{sub:spaces} we describe the $H(\diver)$-conforming VE spaces, whereas in Section \ref{sub:forms} we present the discrete forms. Finally, in Section~\ref{sub:vem problem} we introduce the semi-discrete VE discretization of Problem \ref{pb:wave_weak} and we provide the stability bound of the solution and the energy preserving property of the semi-discrete system.

\subsection{Mesh assumptions and polynomial projections}
\label{sub:mesh}

From now on, we will denote by $E$ a general polygon having $\ell_E$ edges $e$.
For each polygon $E$ and each edge $e$ of $E$ we denote
by $|E|$,
$h_E$ the measure and diameter of $E$
respectively, by $h_e$  we denote the length  of $e$.
%
Furthermore $\nn_E^e$ (resp. $\nn_E)$ denotes the unit outward normal vector to $e$ (resp. to $\partial E$).

Let $\{\Omega_h\}_h$ be a sequence of decompositions of $\Omega$ into general polygons $E$,
where the granularity $h$ is defined as $h = \sup_{E \in \Omega_h} h_E$.
We suppose that $\{\Omega_h\}_h$ fulfills the following assumption:\\
\textbf{(A1) Mesh assumption.}
There exists a positive constant $\rho$ such that for any $E \in \{\Omega_h\}_h$
\begin{itemize}
\item Any $E \in \{\Omega_h\}_h$ is star-shaped with respect to a ball $B_E$ of radius $ \geq\, \rho \, h_E$;
\item Any edge $e$ of any $E \in \{\Omega_h\}_h$,  $h_e \geq\, \rho \, h_E$.
\end{itemize}
We remark that the hypotheses above, though not too restrictive in many practical cases, could possibly be further relaxed, combining the present analysis with the studies in~\cite{BLR:2017,brenner-guan-sung:2017,BdVV:2022}.

Referring to Problem \ref{pb:wave_weak}, we assume that for any $h$ the decomposition $\Omega_h$ matches with the subdivision of $\partial \Omega$ into $\Gamma_D$, $\Gamma_N$, $\Gamma_R$ and  with the definition of the piece-wise constant velocity $c$. We denote by $\Sigma_h$ the set of all the mesh edges and,  for any $E \in \Omega_h$, we define $\Sigma_h^E$ the set of the edges of $E$.
The total number of vertexes, edges and elements in the decomposition $\Omega_h$ are denoted  by $L_V$, $L_e$ and $L_P$, respectively.

Using standard VE notations, for any mesh object $\omega \in \Omega_h \cup \Sigma_h$ and for any
$n \in \mathbb{N}$  let us introduce the space $\Pk_n(\omega)$ to be the space of polynomials defined on $\omega$ of degree $\leq n$,  with the extended notation $\Pk_{-1}(\omega)=\{ 0 \}$.
For any $n \in \mathbb{N}$ and for any non-negative $s\in \R$  we consider the broken spaces
\begin{itemize}
\item $\Pk_n(\Omega_h) = \{q \in L^2(\Omega) \quad \text{s.t.} \quad q|_E \in  \Pk_n(E) \quad \text{for all $E \in \Omega_h$}\}$,

\item $H^s(\Omega_h) = \{v \in L^2(\Omega) \quad \text{s.t.} \quad v|_E \in  H^s(E) \quad \text{for all $E \in \Omega_h$}\}$,
\end{itemize}
equipped with the broken norm and seminorm
\[
\begin{aligned}
&\Vert v\Vert^2_{H^s(\Omega_h)} = \sum_{E \in \Omega_h} \Vert v\Vert^2_{H^s(E)}\,,
\quad
&\vert v\vert^2_{H^s(\Omega_h)} = \sum_{E \in \Omega_h} \vert v\vert^2_{H^s(E)}\,.
\end{aligned}
\]

\noindent
For any $E \in \Omega_h$,  let us introduce  the $\boldsymbol{L^2}$\textbf{-projection} $\Pi_n^{0, E} \colon L^2(E) \to \Pk_n(E)$, given by
\begin{equation*}
\int_{E} q_n (v - \, {\Pi}_{n}^{0, E}  v) \, {\rm d} E = 0 \qquad  \text{for all $v \in L^2(E)$  and $q_n \in \Pk_n(E)$,}
\end{equation*}
with obvious extension for vector functions $\Pi^{0, E}_{n} \colon [L^2(E)]^2 \to [\Pk_n(E)]^2$.
The global counterpart
$\Pi_n^{0} \colon L^2(\Omega_h) \to \Pk_n(\Omega_h)$,
is defined for all $E \in \Omega_h$ by
\begin{equation}
\label{eq:proj-global}
(\Pi_n^{0} v)|_E = \Pi_n^{0,E} v \,.
\end{equation}

In the following the symbol $\lesssim$ will denote a bound up to a generic positive constant,
independent of the mesh size $h$ and of the time step $\tau$ introduced in Section \ref{sec:time_integration}, but which may depend on
$\Omega$, on the ``polynomial" order  $k$ and on the regularity constant appearing in the mesh Assumption \textbf{(A1)}.

\subsection{Virtual Element spaces}
\label{sub:spaces}

We start by presenting an overview of the $H(\diver)$-conforming Virtual Element space, cf. \cite{BeiraodaVeiga2014b,BeiraoVeiga2016,dassi2021}.
%

Let $k \geq 0$ be the  ``polynomial'' order of the method.
We thus consider on each polygonal element $E \in \Omega_h$ the  virtual space
\begin{equation*}
\begin{aligned}
\VV_k(E) = \biggl\{
\vv \in H(\diver, E) \cap H(\rot, E) \,\,\, \text{s.t.} \,\,\,
(i) & \, \,  \diver    \vv  \in \Pk_{k}(E) \,,
\\
(ii)& \, \,  \rot    \vv  \in \Pk_{k-1}(E) \,,
\\
(iii) &
\, \, (\vv \cdot \nn_E^e)|_e \in \Pk_{k}(e) \,\,\, \forall e \in \Sigma^E_h \, \biggr\} \,.
\end{aligned}
\end{equation*}
In the following, we summarize the main properties of the space $\VV_k(E)$. We refer to \cite{BeiraodaVeiga2014b,BeiraoVeiga2016,dassi2021} for a detailed analysis.

\begin{itemize}
\item [\textbf{(P1)}] \textbf{Polynomial inclusion:} $\Pk_k(E) \subseteq \VV_k(E)$;

\item [\textbf{(P2)}] \textbf{Degrees of freedom:}
the following linear operators $\mathbf{D_V}$ constitute a set of DoFs for $\VV_k(E)$: for any $\ww \in \VV_k(E)$ we consider
\begin{itemize}
\item[$\mathbf{D_V1}$] the element moments of the divergence
    \[
    \int_E (\diver \ww) \, p_k \, {\rm d}E
    \qquad \forall p_k \in \Pk_k(E)\setminus \R,
\]
\item[$\mathbf{D_V2}$] the element moments
    \[
    \int_E \ww \cdot (p_{k-1} \xx^\perp) \, {\rm d}E
    \qquad \forall p_k \in \Pk_{k-1},
    \]
    where $\xx^\perp:= \left(y, -x\right)^{\rm T}$;
\item[$\mathbf{D_V3}$] the edge moments
    \[
    \int_e (\ww \cdot \nn_E^e) \, p_k \,{\rm d}e
    \qquad \forall p_k \in \Pk_k(e), \,\, \forall e \in \Sigma_h^E \,.
    \]
\end{itemize}
Therefore the dimension of $\VV_k(E)$ is
\[
\dim(\VV_k(E)) = \ell_E \, (k+1) + k^2 + 2k   \,.
\]
\item [\textbf{(P3)}] \textbf{Computable quantities:}
for any $\ww \in \VV_k(E)$ the DoFs $\mathbf{D_V}$ allow to compute
\[
\Pi^{0,E}_k \ww \,,
\qquad
\diver \ww \,,
\qquad
(\ww \cdot {\nn_E^e})|_e \quad \forall e \in \Sigma_h^E\,.
\]
\end{itemize}
The global virtual element space $\VV_k$ is defined by gluing together all local spaces, that is:
we require that for any internal edge $e \in \Sigma^E_h \cap \Sigma^{E'}_h$, shared by $E$ and $E'$,
\[
\ww|_E \cdot \nn_E^e + \ww|_{E^\prime} \cdot \nn^e_{E^\prime} = 0
\qquad \forall \ww \in \VV_k,
\]
that is in accordance with the DoFs definition $\mathbf{D_V1}$.
Therefore we have
\begin{equation}
\label{eq:VVG}
    \VV_k \defeq \{ \vv \in \VV \quad \text{s.t.} \quad \vv|_E \in
    \VV_k(E)\,\,\, \forall E \in \Omega_h \}.
\end{equation}
The dimension of $\VV_k$ is thus given by
\[
\dim(\VV_k) =  (k+1)L_e + (k^2 + 2k)L_P  \,.
\]
The discrete pressure space $Q_k$ is given by the piecewise polynomial functions of degree $k$, i.e.
\begin{equation}
\label{eq:QG}
    Q_k \defeq \Pk_k(\Omega_h)\,.
\end{equation}

\subsection{Virtual Element forms}
\label{sub:forms}

The next step in the construction of our method is the definition of a discrete version of the continuous forms in \eqref{eq:continuous_forms}.
Following the usual procedure in the VE setting, we need to construct discrete forms that are computable through the  DoFs.
Notice that in the light of property \textbf{(P3)}, for any $\vv_h, \ww_h \in \VV_k$ and $q_h, p_h \in Q_k$ the quantities
\[
\begin{aligned}
n(p_h, q_h)\,,
\qquad
b(\vv_h, q_h)\,,
\qquad
r(\vv_h, \ww_h)\,,
\qquad
F(q_h)
\end{aligned}
\]
are computable.

Whereas for arbitrary functions in $\VV_k$ the form $m(\cdot, \cdot)$ is not computable since the discrete functions are not known in closed form.
Employing property \textbf{(P3)} for any $\vv_h$, $\ww_h \in
\VV_k(E)$ we define the computable local discrete bilinear form:
\begin{equation}
\label{eq:m-VEM}
m_h^{E}(\vv_h, \ww_h) \defeq
\int_E (\Pi_{k}^{0,E} \vv_h) \cdot (\Pi_{k}^{0,E} \ww_h) \, {\rm d}E + h_E^{-2} \mathcal{S}^E(\vv_h, \ww_h) \,.
\end{equation}
The stabilizing term in \eqref{eq:m-VEM} is given by
\begin{equation*}
\mathcal{S}^E(\vv_h, \ww_h) =
S^E \bigl( (I - \Pi^{0,E}_k) \vv_h, \, (I - \Pi^{0,E}_k) \ww_h \bigr) \,,
\end{equation*}
where $S^E(\cdot, \cdot) \colon \VV_k(E) \times \VV_k(E) \to \R$ is a computable symmetric discrete form.
In the present paper we consider the so-called \texttt{dofi-dofi} stabilization \cite{BeiraodaVeiga2013a} defined as follows: let $\vec{\vv}_h$ and $\vec{\ww}_h$ denote the real valued vectors containing the values of the local degrees of freedom (properly scaled) associated to $\vv_h$, $\ww_h$ in the space $\VV_k(E)$ then
\begin{equation*}
S^E(\vv_h, \ww_h) =\vec{\vv}_h \cdot \vec{\ww}_h \,.
\end{equation*}
Under  mesh Assumption \textbf{(A1)} the form $S^E(\cdot, \cdot)$ satisfies the following bounds
(we refer to \cite{dassi2021} for the details)
\begin{equation}
\label{eq:stab2}
\Vert \vv_h \Vert^2_{0,E} \lesssim h_E^2 S^E(\vv_h, \vv_h) \lesssim \Vert \vv_h \Vert^2_{0,E} \,,
\end{equation}
for all $\vv_h \in \VV_k(E) \cap \ker(\Pi^{0,E}_k)$.

The global form $m_h(\cdot, \cdot) \colon
(\VV_k \cup [\Pk_k(\Omega_h)]^2)\times (\VV_k \cup [\Pk_k(\Omega_h)]^2) \to \R$ can be derived adding the local contributions
\begin{equation}
\label{eq:mG-VEM}
m_h(\vv_h, \ww_h) \defeq
\sum_{E \in \Omega_h} m_h^E(\vv_h, \ww_h)
\qquad \forall (\vv_h, \ww_h) \in (\VV_k \cup [\Pk_k(\Omega_h)]^2)
\,.
\end{equation}
Notice that, from \eqref{eq:m-VEM} and \eqref{eq:stab2}, it follows that
\begin{align}
\label{eq:cons}
&m_h(\bm p_k, \vv_h) = m(\bm p_k, \vv_h)
& &\forall \bm p_ k \in [\Pk_k(\Omega_h)]^2\,,
\vv_h \in (\VV_k \cup [\Pk_k(\Omega_h)]^2) \,,
\\
\label{eq:stab3}
&\Vert \vv_h \Vert_{0, \Omega}^2 \lesssim
m_h(\vv_h, \vv_h) \lesssim
\Vert \vv_h \Vert_{0, \Omega}^2
& &\forall \vv_h \in (\VV_k \cup [\Pk_k(\Omega_h)]^2) \,.
\end{align}
\subsection{Virtual Element semi-discrete problem}
\label{sub:vem problem}

Referring to the spaces \eqref{eq:VVG} and \eqref{eq:QG}, the
forms \eqref{eq:continuous_forms} and the
discrete bilinear form \eqref{eq:mG-VEM}, we can state the following semi-discrete problem.

\begin{problem}[VEM problem]\label{pb:wave_vem}
Find  $\uu_h \in L^2(0,T; \VV_k) \cap C^0(0,T; Q^2)$,
and $p_h \in L^2(0,T; Q_k) \cap C^0(0,T; Q_k)$
s.t.
    \begin{align*}
       \left \{
        \begin{aligned}
            m_h({\uu_h}_t, \bm{v}_h) + b(\bm{v}_h, p_h) + r(\uu_h, \bm{v}_h) &= 0 && \text{$\forall \bm{v}_h \in \VV_k$},
            \\
            n({p_h}_t, q_h) - b(\bm{u}_h, q_h) &= F(q_h) &&
            \text{$\forall q_h \in Q_k$},
        \end{aligned}
         \right.
    \end{align*}
with initial condition $\bm u_h(\cdot, 0) = \Pi^0_k \, \bm u_0$ and $p_h(\cdot,0)= \Pi^0_k \, p_0$ in $\Omega$.
\end{problem}

\noindent
The well-posedness of Problem \ref{pb:wave_vem} follows from:
\begin{itemize}
\item discrete inf-sup condition \cite{BeiraoVeiga2016}:
there exists $\widehat{\beta}>0$ s.t.
\[
\inf_{q_h \in Q_k} \sup_{\vv_h \in \VV_k}
\frac{b(\vv_h, q_h)}{\Vert \vv_h \Vert_{\VV} \Vert q_h \Vert_Q}
\geq \widehat{\beta} \,.
\]
\item coercivity on the discrete kernel: the bilinear form
$m_h(\cdot, \cdot)$ satisfies
\[
m_h(\vv_h, \vv_h) \gtrsim \Vert \vv_h \Vert_{\VV}^2
\qquad \forall \vv_h \in \boldsymbol{K}_k\,,
\]
where $\boldsymbol{K}_k \defeq  \{ \bm v_h \in \VV_k\quad \text{s.t.} \quad  b(\bm v_h, q_h) = 0 \,\,\, \forall \, q_h \in Q_k\}$.
\end{itemize}
Let us introduce the discrete energy norm
\begin{equation}\label{def:energy_norm_dis}
\|(\bm v_h, q_h)\|^2_{\mathcal{E}_h} := m_h(\bm v_h, \bm v_h) + \| c^{-1} q_h \|^2_{0,\Omega} \quad \forall \; (\bm v_h, q_h) \in
(\VV_k \cup [\Pk_k(\Omega_h)]^2) \times Q_k \,.
\end{equation}
Then, using analogous argument to that used for the continuous case, the discrete solution $(\uu_h, p_h)$ of Problem \ref{pb:wave_vem}
satisfies the following stability estimate
\begin{equation*}
\sup_{0\leq t \leq T} \|(\bm u_h, p_h)(t)\|_{\mathcal{E}_h}   \leq  \int_{0}^T \norm{c f(s)}_{0,\Omega} \, {\rm d}s + \|(\Pi^0_k \, \bm u_0, \Pi^0_k \, p_0)\|_{\mathcal{E}} \,,
\end{equation*}
where, from \eqref{eq:cons}, we have used that
\[
\|(\bm p_k, q_h)\|_{\mathcal{E}_h} =
\|(\bm p_k, q_h)\|_{\mathcal{E}}
\qquad
\forall (\bm p_k, q_h) \in  [\Pk_k(\Omega_h)]^2 \times Q_k \,.
\]
Furthermore if $f = 0$ and $\Gamma_R = \emptyset$, Problem \ref{pb:wave_vem} is energy conservative, i.e. the solution $(\uu_h, p_h)$ of Problem \ref{pb:wave_vem} satisfies
\begin{equation}
\label{eq:cons-sdisc}
\|(\bm u_h, p_h)(t)\|_{\mathcal{E}_h} =
\|(\Pi^0_k \, \bm u_0, \Pi^0_k \, p_0)\|_{\mathcal{E}}
\qquad
\forall t \in [0,T] \,,
\end{equation}
that is the semi-discrete counterpart of \eqref{eq:cons-cont}.

\begin{remark}
Notice that definitions \eqref{def:energy_norm} and \eqref{def:energy_norm_dis} and bounds \eqref{eq:stab3} imply the following norm equivalence
\begin{equation}
\label{eq:stab4}
\|(\vv_h, q_h)\|_{\mathcal{E}}   \lesssim
\|(\vv_h, q_h)\|_{\mathcal{E}_h} \lesssim
\|(\vv_h, q_h)\|_{\mathcal{E}}   \qquad
\forall \; (\bm v_h, q_h) \in
\VV_k \times Q_k \,.
\end{equation}
\end{remark}

\begin{remark}
\label{rm:3d}
The proposed approach can be easily extended to more general situations such as the three dimensional case \cite{BeiraodaVeiga2014b}, and domains with curved boundary/interfaces \cite{dassi2021,bend}.
The analysis could be developed with very similar arguments to the ones in the forthcoming sections.
\end{remark}


\section{Theoretical analysis} \label{sec:theory}

In this section, we present some theoretical results for the virtual element Problem \ref{pb:wave_vem}.
In Section \ref{sub:int} we review the interpolation estimates, whereas in Section \ref{sub:convergence} and Section \ref{sub:energy}
we provide the convergence analysis and the energy error estimate respectively.

\subsection{Interpolation estimates}
\label{sub:int}
We now recall the optimal approximation properties for the space
$\VV_k$ (see \cite{dassi2021}).
We define the linear Fortin operator $\Pi_k^{F} \colon [H^1(\Omega)]^2 \to \VV_k$ in the following way.
For any $\ww \in [H^1(\Omega)]^2$ and for all $e \in \Sigma_h$ and $E \in \Omega_h$,  $\Pi_k^{F}\ww$ is determined by
\begin{eqnarray*}
\hspace{-1em}\int_E \diver (\ww - \Pi_k^{F} \ww) \, p_k \, {\rm d}E = 0
&\forall&\hspace{-1em}p_k \in \Pk_k(E)\setminus \R\,,
\\
\hspace{-1em}\int_E \rot (\ww - \Pi_k^{F} \ww) \, p_{k-1} \, {\rm d}E = 0
&\forall&\hspace{-1em}p_{k-1} \in \Pk_{k-1}(E)\,,
\\
\hspace{-1em}\int_e (\ww - \Pi_k^{F} \ww) \cdot \nn^e p_k \, {\rm
d}e = 0
&\forall&\hspace{-1em}p_k \in \Pk_k(e)\,.
\end{eqnarray*}
The conditions above implies that the following diagram is commutative, i.e.,
\begin{equation*}
\begin{split}
[H^1(\Omega)]^2
\,
&\xrightarrow[]{  \,\,\,\,\, \text{{$\diver$}} \,\,\,\,\,  }
\, \,
Q
\, \,
\xrightarrow[]{\, \, \,\,\,\,\, 0 \,\,\,\,\, \, \,}
\,
0
\\
\Pi_k^F  \bigg\downarrow \qquad &
\quad \quad \,\,\,
\Pi_0^k        \bigg\downarrow
\\
\VV_k
\quad \,\,\,
&\xrightarrow[]{  \,\,\,\,\, \text{{$\diver$}} \,\,\,\,\,  }
\,
Q_k
\,
\xrightarrow[]{\, \, \,\,\,\,\, 0 \,\,\,\,\, \, \,}
\,
0
\end{split}
\end{equation*}
where $0$ is the map that associates to every function the value $0$. In particular, for any $\ww \in [H^1(\Omega)]^2$ it holds
(see \cite{dassi2021})
\begin{equation}
\label{eq:divpf}
\begin{gathered}
\diver (\Pi_k^F \ww) = \Pi^0_k \diver \ww \,,\quad
\curl (\Pi_k^F \ww) = \Pi^0_{k-1} \curl \ww \,,
\\
((\Pi_k^F \ww) \cdot \nn^e)|_e = (\Pi^{0,e}_k (\ww \cdot \nn^e))|_e
\quad \forall e \in \Sigma_h \,.
\end{gathered}
\end{equation}
For the Fortin operator  we have the following interpolation estimate (see \cite[Proposition 4.1]{dassi2021}).
\begin{prop}[Approximation property of $\VV_k$]
\label{prp:int VDG}
Under  Assumption \cfan{(A1)} for any $\vv \in \VV \cap [H^{s}(\Omega_h)]^2$ where $1 \leq s \leq k+1$,  $\Pi^k_{ F} \vv$ satisfies the estimate
\[
\Vert \vv - \Pi^k_{ F} \vv \Vert_{0,\Omega}
\lesssim h^{s} |v|_{s,\Omega_h} \,.
\]
\end{prop}

We now review a classical approximation result for polynomials on star-shaped domains, see for instance \cite{brenner-scott:book}.

\begin{lemma}[Bramble-Hilbert]
\label{lm:bramble}
Under Assumption \cfan{(A1)} and referring to \eqref{eq:proj-global}, for all $q \in H^{s}(\Omega_h)$ where $0 \leq s \leq k+1$,  it holds
\[
\|q -  \Pi^0_k q\|_{\Omega, 0} \lesssim h^{s} \, |q|_{\Omega_h, s} \,.
\]
\end{lemma}

\subsection{Convergence analysis}
\label{sub:convergence}
In this section we provide the convergence property for the semi-discrete scheme.

\begin{prop}\label{thm:convergence}
Under  Assumption \cfan{(A1)}, let $(\uu, p)$ be the solution of Problem \ref{pb:wave_weak} and
$(\uu_h, p_h)$ be the solution of Problem \ref{pb:wave_vem}.
Assume that
\[
\uu_t, \, p_t \in L^1(0,T; H^{k+1}(\Omega_h))
\qquad \text{and} \qquad
\uu_0, \, p_0 \in H^{k+1}(\Omega_h)\,.
\]
Then for all $t \in (0,T)$  the following error estimate holds:
\begin{multline*}
\|(\bm u - \bm u_h, p-p_h)(t)\|_\mathcal{E} \lesssim
h^{k+1}
\bigl( | \bm u_0|_{k+1, \Omega_h}  + \widehat{c}\,| p_0|_{k+1, \Omega_h} +
| \uu_t|_{L^1(0, t; H^{k+1}(\Omega_h))}  +
\widehat{c}\,| p_t|_{L^1(0, t; H^{k+1}(\Omega_h))}
 \bigr) \,,
\end{multline*}
where $\widehat{c} := \Vert c^{-1} \Vert_{L^{\infty}(\Omega \times (0,T))}$.
\end{prop}

\begin{proof}
For all $t \in (0,T)$, let us introduce the following error quantities
\[
\begin{aligned}
{\bm e}_{I}(t) &:=
\uu(t) - \Pi_k^F\uu(t)\,,
\qquad  &
\rho_{I}(t) &:= p(t)- \Pi_k^0 p(t)\,,
\\
{\bm e}_h(t) &:= \Pi_k^F\uu(t) - \uu_h(t)\,,
\qquad  &
\rho_h(t) &:= \Pi_k^0 p(t) - p_h(t)\,,
\end{aligned}
\]
From triangle inequality and \eqref{eq:stab4} it holds
\begin{equation}\label{eq:errore_0}
\|(\bm u - \bm u_h, p-p_h)(t)\|_\mathcal{E}
\lesssim \|({\bm e}_{I}, \rho_{I})(t)\|_\mathcal{E} +
\|({\bm e}_h, \rho_h)(t)\|_{\mathcal{E}_h}
 \,.
\end{equation}
The first term on the right-hand side of \eqref{eq:errore_0} can be bounded by using Proposition \ref{prp:int VDG} and Lemma \ref{lm:bramble} getting
\begin{equation}\label{eq:stima_errore_1}
\begin{aligned}
\|({\bm e}_{I}, \rho_{I})(t)\|_\mathcal{E}
&\lesssim
h^{(k+1)}  \left( | \bm u(t)|_{k+1, \Omega_h}  + \widehat{c}\, | p(t)|_{k+1, \Omega_h} \right)
\\
& \lesssim
h^{k+1}
\left( | \bm u_0|_{k+1, \Omega_h}  + \widehat{c}\,| p_0|_{k+1, \Omega_h} +
\int_0^t \left( | \uu_t(s)|_{k+1, \Omega_h}  +
\widehat{c}\,| p_t(s)|_{k+1, \Omega_h} \right) \, {\rm d}s
 \right)
\\
& =
h^{k+1}
\bigl( | \bm u_0|_{k+1, \Omega_h}  + \widehat{c}\,| p_0|_{k+1, \Omega_h} +
| \uu_t|_{L^1(0, t; H^{k+1}(\Omega_h))}  +
\widehat{c}\,| p_t|_{L^1(0, t; H^{k+1}(\Omega_h))}
 \bigr) \,.
\end{aligned}
\end{equation}
To estimate the second term on the right-hand side of \eqref{eq:errore_0} we proceed as follows. We consider Problem \ref{pb:wave_weak} and Problem \ref{pb:wave_vem}
and obtain
\[
\begin{aligned}
 m(\uu_t, \bm{e}_h) + b(\bm{e}_h, p) + r(\bm u, \bm{e}_h) +
 n(p_t, \rho_h) - b(\bm{u}, \rho_h) &= F(\rho_h),
\\
 m_h({\bm{u}_h}_t, \bm{e}_h) + b(\bm{e}_h, p_h) + r(\uu_h, \bm{e}_h) + n({p_h}_t, \rho_h) - b(\bm{u}_h, \rho_h) &= F(\rho_h).
\end{aligned}
\]
Then subtracting the previous equations we get the following error equation
\begin{equation}\label{eq:error_equation}
m(\uu_t, \bm{e}_h) -  m_h({\bm{u}_h}_t, \bm{e}_h) +
n((p  - p_h)_t, \rho_h) +
r(\bm u - \bm u_h, \bm{e}_h) +
\eta_b   = 0 \,,
\end{equation}
where
\[
\eta_b := b(\bm{e}_h, p - p_h) + b(\bm{u}_h-\bm{u}, \rho_h).
\]
First, we notice that the term $\eta_b$ vanishes, in fact
\begin{equation}
\label{eq:eta_b}
\begin{aligned}
\eta_b   & =
b(\Pi^F_k \bm u, p - p_h) - b(\bm u_h, p - p_h) +
b(\bm{u}_h, \Pi^0_k p - p_h) - b(\bm{u}, \Pi^0_k p - p_h)
\\
& =
b(\Pi^F_k \bm u, p - p_h) +
b(\bm{u}, p_h - \Pi^0_k p ) +
b(\bm{u}_h, \Pi^0_k p - p)
\\
&
\begin{aligned}
&=
b(\Pi^F_k \bm u, p - p_h) +
b(\bm{u}, p_h - \Pi^0_k p )
\quad &
& \text{($\diver \uu_h \in \Pk_k(\Omega_h)$ \& def. $\Pi^0_k$)}
\\
& =
b(\Pi^F_k \bm u, p - p_h) +
b(\Pi^F_k \bm{u}, p_h - \Pi^0_k p )
\quad &
& \text{($ p_h - \Pi^0_k p \in \Pk_k(\Omega_h)$ \&  \eqref{eq:divpf})}
\\
& =
b(\Pi^F_k \bm{u}, p - \Pi^0_k p )
\\
& =
0.
\quad &
& \text{($\diver (\Pi^F_h\uu) \in \Pk_k(\Omega_h)$ \& def. $\Pi^0_k$)}
\end{aligned}
\end{aligned}
\end{equation}
Therefore, from \eqref{eq:error_equation} and \eqref{eq:eta_b} we infer
\begin{equation}\label{eq:conv1}
m_h({\bm{e}_h}_t, \bm{e}_h) +
n({\rho_h}_t, \rho_h) +
r(\bm e_h, \bm{e}_h)
=
(m_h(\Pi^F_k(\uu_t), \bm{e}_h) -  m(\uu_t, \bm{e}_h)) -
n({\rho_{I}}_t, \rho_h)  -
r(\bm e_{{I}}, \bm{e}_h) \,.
\end{equation}
Recalling that $(\bm{e}_h \cdot \nn^e)|_e \in \Pk_k(e)$ for all $e \in \Sigma_h$, and employing \eqref{eq:divpf} we have
\begin{equation}
\label{eq:conv2}
r(\bm e_{I}, \bm{e}_h) =
r(\bm u - \Pi^F_k \bm u, \bm{e}_h) = 0 \,.
\end{equation}
Therefore, since $r(\bm e_h, \bm{e}_h) \geq 0$, from \eqref{eq:conv1} and \eqref{eq:conv2}, and recalling definition \eqref{def:energy_norm_dis} we obtain
\begin{equation}
\label{eq:conv3}
\frac{1}{2}\frac{d}{dt}\Vert (\bm e_h, \bm e_h) \Vert_{\mathcal{E}_h}^2
\leq
(m_h(\Pi^F_k(\uu_t), \bm{e}_h) -  m(\uu_t, \bm{e}_h)) -
n({\rho_{I}}_t, \rho_h) =: \eta_{\uu} + \eta_p \,.
\end{equation}
The first term on the right-hand side of the previous equation can be bounded as follows
\begin{equation}\label{eq:conv4}
\begin{aligned}
\eta_{\uu} &=
m_h(\Pi^F_k(\uu_t) - \Pi^0_k(\uu_t), \bm{e}_h) +  m(\Pi^0_k(\uu_t) - \uu_t, \bm{e}_h)
& & \text{(by \eqref{eq:cons})}
\\
&\lesssim
\left(
\Vert\Pi^F_k(\uu_t) - \Pi^0_k(\uu_t) \Vert_{0, \Omega} +
\Vert \uu_t - \Pi^0(\uu_t) \Vert_{0, \Omega}
\right)
\Vert \bm e_h \Vert_{0, \Omega}
& & \text{(by \eqref{eq:stab3})}
\\
&\lesssim
\left(
\Vert \uu_t - \Pi^F_k(\uu_t) \Vert_{0, \Omega} +
\Vert \uu_t - \Pi^0(\uu_t) \Vert_{0, \Omega}
\right)
\Vert \bm e_h \Vert_{0, \Omega}
& & \text{(by tri. ineq.)}
\\
& \lesssim
h^{k+1} \,\vert \uu_t\vert_{H^{k+1}(\Omega_h)}
m_h(\bm e_h, \bm e_h)^{1/2}
\end{aligned}
\end{equation}
where in the last inequality we used Proposition \ref{prp:int VDG},
Lemma \ref{lm:bramble} and equation \eqref{eq:stab3}.
Whereas $\eta_p$ is estimated by
\begin{equation}
\label{eq:conv5}
\begin{aligned}
\eta_{\uu} =
n(\Pi^0(p_t) - p_t, \rho_h)
& \leq
\|c^{-1}(\Pi^0(p_t) - p_t) \|_{0, \Omega}
\|c^{-1} \rho_h \|_{0, \Omega}
\\
&\lesssim
\widehat{c} \,h^{k+1} \,
\vert p_t\vert_{H^{k+1}(\Omega_h)}
\|c^{-1} \rho_h \|_{0, \Omega} \,.
\end{aligned}
\end{equation}
Therefore recalling definition \eqref{def:energy_norm_dis}, from
\eqref{eq:conv3},  \eqref{eq:conv4} and \eqref{eq:conv5} and the Cauchy-Schwarz inequality,
 we obtain
\begin{equation*}
\frac{d}{dt}\Vert (\bm e_h, \rho_h)(t) \Vert_{\mathcal{E}_h}
\lesssim
h^{k+1} \left( \vert \uu_t(t)\vert_{H^{k+1}(\Omega_h)} +
\widehat{c} \,
\vert p_t(t)\vert_{H^{k+1}(\Omega_h)} \right) \,.
\end{equation*}
By integrating the previous bound on $(0,t)$ we obtain
\begin{equation}
\label{eq:conv7}
\begin{aligned}
\Vert (\bm e_h, \rho_h)(t) \Vert_{\mathcal{E}_h}
& \lesssim
\Vert (\bm e_h, \rho_h)(0) \Vert_{\mathcal{E}_h}
+
h^{k+1} \left( \int_0^t \vert \uu_t(s)\vert_{H^{k+1}(\Omega_h)} \,{\rm d}s +
\widehat{c}
\int_0^t \vert p_t(s)\vert_{H^{k+1}(\Omega_h)} \, {\rm d}s\right)
\\
& \lesssim
\Vert (\bm e_h, \rho_h)(0) \Vert_{\mathcal{E}} +
h^{k+1} \left( \Vert \uu_t\Vert_{L^1(0,t; H^{k+1}(\Omega_h))} +
\widehat{c} \,
\Vert p_t\Vert_{L^1(0,t; H^{k+1}(\Omega_h))}\right)
\,.
\end{aligned}
\end{equation}
We finally bound the initial data error
\begin{equation}
\label{eq:conv8}
\begin{aligned}
\Vert (\bm e_h, \rho_h)(0) \Vert_{\mathcal{E}}
&\lesssim
\Vert \Pi^F_k \uu_0 - \Pi^0_k \uu_0 \Vert_{0, \Omega}
+
\widehat{c}
\Vert p_0 - \Pi^0_k p_0 \Vert_{0, \Omega}
\\
&\lesssim
h^{k+1} \left( \vert \uu_0\vert_{H^{k+1}(\Omega_h)} +
\widehat{c} \,
\vert p_0\vert_{H^{k+1}(\Omega_h)}\right)  \,.
\end{aligned}
\end{equation}
The thesis now follows combining \eqref{eq:conv7}, \eqref{eq:conv8} and \eqref{eq:stima_errore_1} in \eqref{eq:errore_0}.
\end{proof}

\subsection{Energy error}
\label{sub:energy}

An important aspect in wave propagation problems is the energy conservation of the continuous and semi-discrete system (cf.  \eqref{eq:cons-cont} and \eqref{eq:cons-sdisc} rispectively).
In the following proposition we estimate the errors between the energy of the continuous and the semi-discrete energy of the system.

\begin{prop}
\label{prp:enf}
Under assuption \cfan{(A1)}, let $(\uu, p)$ be the solution of Problem \ref{pb:wave_weak} and
$(\uu_h, p_h)$ be the solution of Problem \ref{pb:wave_vem} with $f=0$ and $\Gamma_R=\emptyset$.
Assume that
$\uu_0, \, p_0 \in H^{k+1}(\Omega_h)$.
Then for all $t \in (0,T)$, the following  estimate holds:
\begin{equation*}
\begin{aligned}
0 \leq \|(\bm u , p)(t)\|_\mathcal{E}^2 -
\|(\bm u_h, p_h)(t)\|_{\mathcal{E}_h}^2
\lesssim
h^{2(k+1)}
\bigl( | \bm u_0|_{k+1, \Omega_h}^2  + \widehat{c}^2\,| p_0|_{k+1, \Omega_h}^2 \bigr) \,,
\end{aligned}
\end{equation*}
where $\widehat{c} := \Vert c^{-1} \Vert_{L^{\infty}(\Omega \times (0,T))}$.
\end{prop}

\begin{proof}.
From \eqref{eq:cons-cont} and \eqref{eq:cons-sdisc} we infer
\[
\begin{aligned}
\|(\bm u , p)(t)\|_\mathcal{E}^2 -
\|(\bm u_h, p_h)(t)\|_{\mathcal{E}_h}^2  & =
 \|(\bm u_0 , p_0)\|_\mathcal{E}^2 -
\|(\Pi^0_k \bm u_0, \Pi^0_k p_0)\|_{\mathcal{E}}^2
\\
& =
\Vert \bm u_0 \Vert_{0,  \Omega}^2 +
\Vert \bm c^{-1}p_0 \Vert_{0,  \Omega}^2 -
\Vert \bm \Pi^0_k \bm u_0 \Vert_{0,  \Omega}^2 -
\Vert \bm c^{-1} \Pi^0_k p_0 \Vert_{0,  \Omega}^2 \,.
\end{aligned}
\]
Direct application of Pythagorean Theorem yields
\[
\Vert \zeta \Vert_{0,  \Omega}^2  =
\Vert \Pi^0_k \zeta  \Vert_{0,  \Omega}^2 +
\Vert  \zeta - \Pi^0_k \zeta  \Vert_{0,  \Omega}^2 \,,
\qquad
\forall \zeta \in L^2(\Omega) \,,
\]
therefore, recalling that $c$ is piece-wise constant w.r.t. $\Omega_h$, we infer
\begin{equation}
\label{eq:enf}
\begin{aligned}
\|(\bm u , p)(t)\|_\mathcal{E}^2 -
\|(\bm u_h, p_h)(t)\|_{\mathcal{E}_h}^2  &=
\Vert  \bm u_0 - \Pi^0_k \bm u_0  \Vert_{0,  \Omega}^2 +
\Vert  c^{-1}(p_0 - \Pi^0_k p_0)  \Vert_{0,  \Omega}^2 \,.
\end{aligned}
\end{equation}
Then, from \eqref{eq:enf} and Lemma \ref{lm:bramble}, the following bounds hold
\[
\begin{aligned}
\|(\bm u , p)(t)\|_\mathcal{E}^2 -
\|(\bm u_h, p_h)(t)\|_{\mathcal{E}_h}^2  & \geq 0 \,,
\\
\|(\bm u , p)(t)\|_\mathcal{E}^2 -
\|(\bm u_h, p_h)(t)\|_{\mathcal{E}_h}^2  & \lesssim
h^{2(k+1)}
\bigl( | \bm u_0|_{k+1, \Omega_h}^2  + \widehat{c}^2\,| p_0|_{k+1, \Omega_h}^2 \bigr) \,.
\end{aligned}
\]
\end{proof}

\section{Time integration}\label{sec:time_integration}

In the light of energy conservation \eqref{eq:cons-cont} and \eqref{eq:cons-sdisc} and the energy error estimate in Proposition \ref{prp:enf}, it is crucial to understand how the energy is preserved or not under time-stepping schemes.

In the present section we formulate a fully discrete version of Problem \ref{pb:wave_vem} aiming at preserving the energy of the system.
Therefore, we introduce a sequence of time steps $t_n = n \tau$, $n=0,\dots,N$, with time step size $\tau$.
Next, we define $\vv_{h, \tau}^n := \vv_h(\cdot, t_n)$
(resp. $q_{h, \tau}^n := q_h(\cdot, t_n)$)
as the approximation of the function $\vv_h(\cdot, t) \in \VV_k$
(resp. $q_h(\cdot, t) \in Q_k$ )
at time~$t_n$, $n=0,\dots,N$.
To integrate in time Problem \ref{pb:wave_vem}, we take under consideration the family of $\theta$-method schemes and analyse their energy conservation properties.
The fully discrete systems consequently reads as follows:

\noindent
Given $(\uu_{h, \tau}^0, p_{h, \tau}^0) = (\Pi^0_k \uu_0, \Pi^0_k p_0)$, find
$(\uu_{h, \tau}^n, p_{h, \tau}^n)$ for $n=0, \dots, N$
s.t.
\begin{align}\label{tetametodo}
       \left \{
        \begin{aligned}
            m_h\left(\frac{\bm{u}_{h,\tau}^{n+1}-\bm{u}_{h,\tau}^{n}}{\tau}, \bm{v}_h \right) +
            b(\bm{v}_h, \theta p_{h,\tau}^{n+1}+(1-\theta)p_{h,\tau}^n) +
             r(\theta\bm{u}_{h,\tau}^{n+1} + (1-\theta)\bm u_{h,\tau}^n, \bm{v}_h)  = 0
            \qquad \forall \vv_h \in \VV_k,
            \\
          n\left(\frac{{p}_{h,\tau}^{n+1}-p_{h,\tau}^n}{\tau},q_h\right) - b(\theta\bm{u}_{h,\tau}^{n+1}+(1-\theta)\bm u_{h,\tau}^n, q_h)
            = \theta F^{n+1}(q_h) + (1-\theta) F^{n}(q_h)
            \qquad \quad \,\,\,  \forall q_h \in Q_k,
        \end{aligned}
         \right.
\end{align}
where $F^{n}(q_h) := (f(t_n),q_h)_\Omega$ and $\theta \in [0,1]$. It is well known that all $\theta$-methods are first order accurate in time, except for $\theta=1/2$, i.e., the Crank-Nicolson method, which is second order accurate.

To analyse the energy conservation of the proposed schemes we consider null forcing terms, i.e. $f=0$, and $\Gamma_R=\emptyset$. Thus, the above system reduces to
\begin{align}\label{eq:teta-method_energy}
       \left \{
        \begin{aligned}
            & m_h\left(\frac{\bm{u}_{h,\tau}^{n+1}-\bm{u}_{h,\tau}^{n}}{\tau}, \bm{v}_h \right) +
            b(\bm{v}_h, \theta p_{h,\tau}^{n+1}+(1-\theta)p_{h,\tau}^n)   = 0
            \quad &\forall \vv_h \in \VV_k,
            \\
            & n\left(\frac{{p}_{h,\tau}^{n+1}-p_{h,\tau}^n}{\tau},q_h\right) - b(\theta\bm{u}_{h,\tau}^{n+1}+(1-\theta)\bm u_{h,\tau}^n, q_h) =
            0 \quad   &\forall q_h \in Q_k \,.
        \end{aligned}
         \right.
\end{align}
We start by considering in \eqref{eq:teta-method_energy}
\[
\bm v_h = \theta\bm{u}_{h,\tau}^{n+1}+(1-\theta)\bm u_{h,\tau}^n  \in \VV_k,
\qquad \text{and} \qquad
q_h = \theta p_{h,\tau}^{n+1}+(1-\theta)p_{h,\tau}^n \in Q_k,
\]
and by summing together the two above equation we get
\begin{equation*}
m_h\left(\frac{\bm{u}_{h,\tau}^{n+1}-\bm{u}_{h,\tau}^{n}}{\tau}, \theta\bm{u}_{h,\tau}^{n+1}+(1-\theta)\bm u_{h,\tau}^n \right) + n\left(\frac{{p}_{h,\tau}^{n+1}-p_{h,\tau}^n}{\tau},\theta p_{h,\tau}^{n+1}+(1-\theta)p_{h,\tau}^n\right)
 = 0,
\end{equation*}
that is
\begin{equation*}
 \theta \Vert(\bm{u}_{h,\tau}^{n+1}, p_{h,\tau}^{n+1})\Vert_{\mathcal{E}_h}^2 +
 (1 -2\theta) \left(m_h(\bm{u}_{h,\tau}^{n+1},\bm u_{h,\tau}^{n}) + n(p_{h,\tau}^{n+1},p_{h,\tau}^{n})\right)
=
(1-\theta) \Vert(\bm{u}_{h,\tau}^{n}, p_{h,\tau}^{n})\Vert_{\mathcal{E}_h}^2
\,,
\end{equation*}
which,  rearranging the terms, is
\begin{equation}\label{eq:energy_discrete_eq}
\Vert(\bm{u}_{h,\tau}^{n+1}, p_{h,\tau}^{n+1})\Vert_{\mathcal{E}_h}^2
+
(2\theta - 1)
\Vert(\bm{u}_{h,\tau}^{n+1} - \bm{u}_{h,\tau}^{n}, p_{h,\tau}^{n+1}- p_{h,\tau}^{n})\Vert_{\mathcal{E}_h}^2
=
\Vert(\bm{u}_{h,\tau}^{n}, p_{h,\tau}^{n})\Vert_{\mathcal{E}_h}^2 \,.
\end{equation}
Now, it is easy to see that for $n=0, \dots, N$:
\begin{itemize}
    \item if $\theta = 1/2$, i.e., for the Crank-Nicolson method, the discrete energy is conserved, i.e.
\begin{equation*}
\Vert(\bm{u}_{h,\tau}^{n}, p_{h,\tau}^{n})\Vert_{\mathcal{E}_h}
=
\Vert(\bm{u}_{h,\tau}^{0}, p_{h,\tau}^{0})\Vert_{\mathcal{E}_h}
=
\Vert(\Pi^0_k\bm{u}_0, \Pi^0_k p_0)\Vert_{\mathcal{E}},
\end{equation*}
\item if $1/2 < \theta \leq 1 $, that is $2\theta - 1 >0$, the second term in \eqref{eq:energy_discrete_eq} is positive and then we can  obtain
\begin{equation*}
\Vert(\bm{u}_{h,\tau}^{n+1}, p_{h,\tau}^{n+1})\Vert_{\mathcal{E}_h}^2
\leq
\Vert(\bm{u}_{h,\tau}^{n}, p_{h,\tau}^{n})\Vert_{\mathcal{E}_h}^2 \,,
\end{equation*}
that means dissipation of energy, cf. \cite{Geveci1988},
\item if  $0 \leq \theta < \frac12$, that is $2\theta - 1 < 0$, the second term in \eqref{eq:energy_discrete_eq} is negative and then we have
\begin{equation}\label{eq:Energy_increase}
\Vert(\bm{u}_{h,\tau}^{n+1}, p_{h,\tau}^{n+1})\Vert_{\mathcal{E}_h}^2
\geq
\Vert(\bm{u}_{h,\tau}^{n}, p_{h,\tau}^{n})\Vert_{\mathcal{E}_h}^2 \,,
\end{equation}
that is the method produces a nondecreasing energy at each time step, in agreement with \cite{Glowinski2006}.
\end{itemize}

To obtain an explicit nearby energy-conservative scheme, one can consider, for example, the symplectic Euler scheme, that reduces to:
\begin{align}\label{sympEuler}
       \left \{
        \begin{aligned}
            & n\left(\frac{{p}_{h, \tau}^{n+1}-p_{h, \tau}^n}{\tau},q_h \right) - b(\bm{u}_{h, \tau}^n, q_h) = F^{n}(q_h)
            &\quad & \forall q_h \in Q_k,
            \\ & m_h\left(\frac{\bm{u}_{h, \tau}^{n+1}-\bm{u}_{h, \tau}^{n}}{\tau}, \bm{v}_h\right) + b(\bm{v}_h, p_{h, \tau}^{n+1}) + r(\bm{u}_{h, \tau}^{n}, \bm{v}_h)  = 0
            &\quad & \forall \vv_h \in \VV_k \,.
        \end{aligned}
         \right.
\end{align}
We refer the reader to \cite{Kirby2015} for the analysis.

\newcommand{\dE}{\text{d}E}
\newcommand{\Or}[1]{O\left(h^{#1}\right)}

\section{Numerical tests}
\label{sec:numExe}
In this section we provide some numerical examples to show the behaviour of the method and give numerical evidence of the theoretical results derived in the previous sections.
We consider three test cases: the first in Subsection \ref{subsec:numexe1}
presents the expected convergence rates for different approximation degree over several mesh families.
The second test case, given in Subsection \ref{subsec:numexe2},
is focused on the energy conservation properties of the scheme.
Finally, in Subsection~\ref{subsec:scattering},
we consider a wave propagation problem in a domain having with small curved inclusions. The latter
is handled by using the extension of VEM to curved edges (cf. Remark \ref{rm:3d}).

\subsection{Error decay}\label{subsec:numexe1}

In this first test case, we verify the  expected convergence rate of the method
for different mesh families:
\begin{itemize}
    \item \texttt{tria} a simplicial mesh;
    \item \texttt{quad} a Cartesian mesh;
    \item \texttt{hexa} composed by distorched hexagons;
    \item \texttt{voro} a mesh made of Voronoi cells optimized via a Lloyd algorithm.
\end{itemize}

In order to compute the VEM errors between the exact solution $(\uu, p)$ and
the VEM solution $(\uu_h, p_h)$ at the final time $T$, we consider the computable $L^2$-like error quantities,
\begin{gather*}
    e_{\bm{u}} \defeq \Vert \uu(T) - \Pi^0_k \uu_h(T) \Vert_{0, \Omega}
    \quad \text{and} \quad
    e_p  \defeq \Vert p(T) - \Pi^0_k p_h(T) \Vert_{0, \Omega}.\,
\end{gather*}
Let us set the mesh-size parameter
\begin{equation*}
    h := \frac{1}{L_P}\sum_{E\in\Omega_h} h_E\,.
\end{equation*}
For each family, we build a sequence of four meshes with decreasing mesh size $h$.  In Figure~\ref{meshUsedConv} we depict one mesh as representative of each family.

\begin{figure}[!htb]
    \centering
    \begin{tabular}{cc}
        \texttt{tria} &  \texttt{quad} \\
        \includegraphics[width=0.31\textwidth]{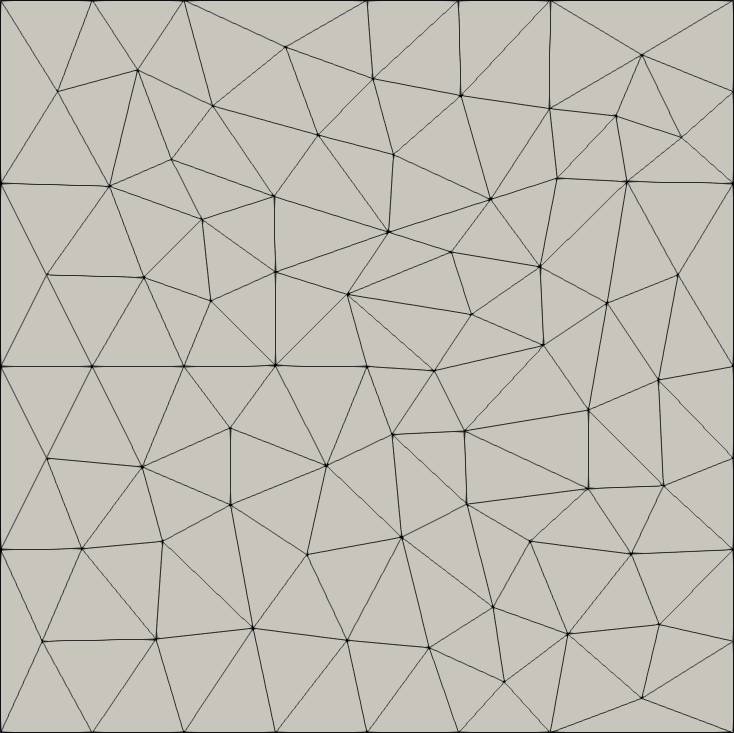}&
        \includegraphics[width=0.31\textwidth]{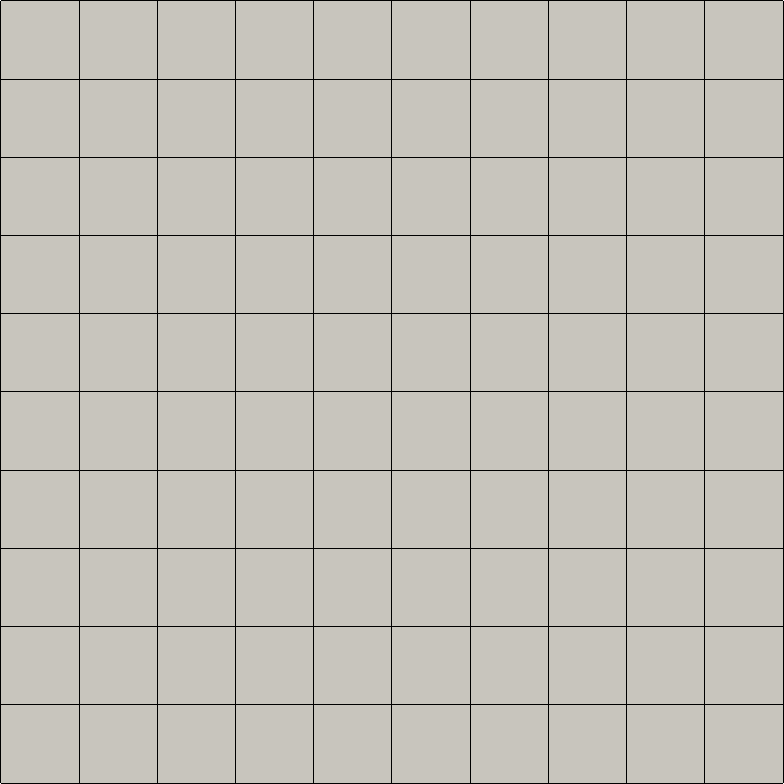}\\
        \texttt{hexa} &  \texttt{voro} \\
        \includegraphics[width=0.31\textwidth]{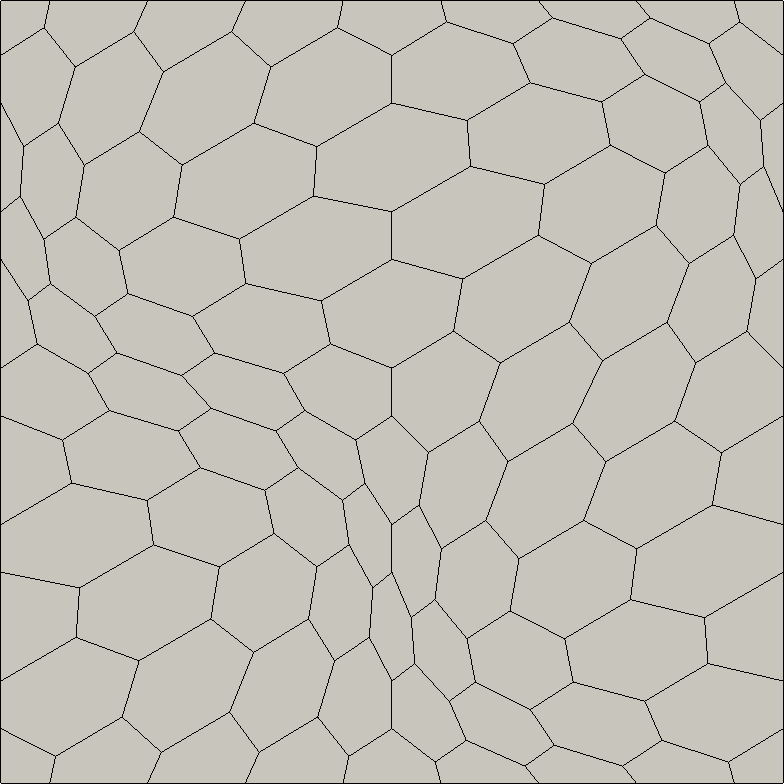}&
        \includegraphics[width=0.31\textwidth]{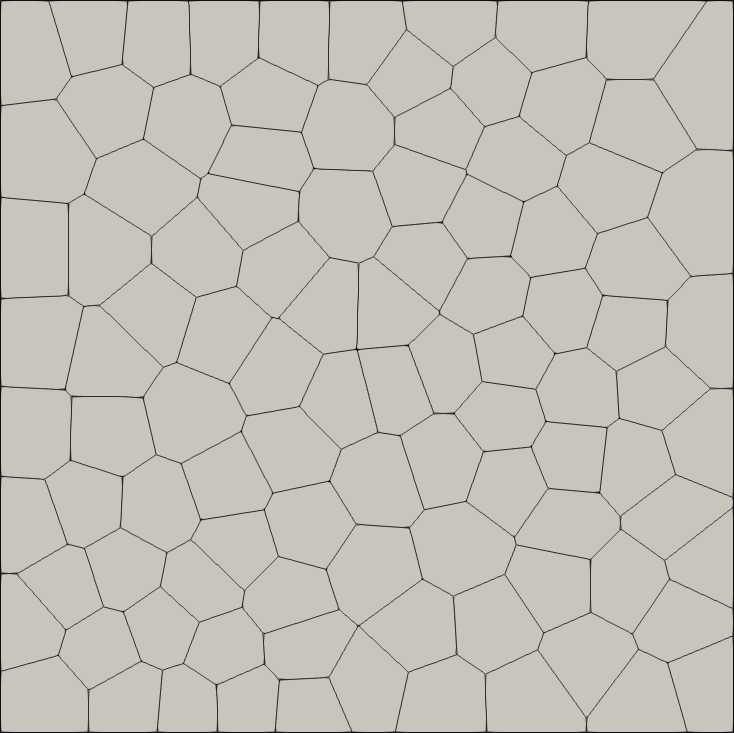}\\
    \end{tabular}
    \caption{{Example of the adopted polygonal meshes}. Example in Subsection~\ref{subsec:numexe1}.}
    \label{meshUsedConv}
\end{figure}

In accordance with Proposition \ref{thm:convergence},
the trend of each error indicator is computed and compared to the
expected convergence trend. We consider a problem on $\Omega = (0,1)^2$
with analytical solution given by
\[
\begin{aligned}
    \bm{u}(\bm{x}, t) &=
    \begin{bmatrix}
        -2\pi\cos(2\pi x)\cos(2\pi y)\sin(t)\\
        2\pi\sin(2\pi x)\sin(2\pi y)\sin(t)
    \end{bmatrix},
\\
    p(\bm{x}, t) &= \cos(2 \pi y) \sin(2 \pi x) \cos(t).\,
\end{aligned}
\]
We set $\Gamma_D=\partial \Omega$, $\Gamma_N=\Gamma_R=\emptyset$ and choose $c=1$. Boundary and initial conditions as well as the forcing term $f$ in Problem \ref{pb:wave_model} are computed accordingly.
Since we are considering the space discretization error, we set $T=1.e-7$ and $\tau=1.e-8$ with $\theta=1$ in \eqref{tetametodo}.

The computed errors are given in Figure~\ref{fig:case1_error_1} and~\ref{fig:case1_error_2}.
Such convergence lines are coherent with the estimates derived in the theoretical results,
see Proposition~\ref{thm:convergence}.
For high approximation degree and fine mesh size,
we notice stagnation of the error of the \texttt{hexa} and \texttt{voro} families.
This is an expected behaviour discussed in \cite{Berrone2017}.
A possible solution is to introduce a different basis for the polynomial expansion
that makes the local systems better conditioned. However, this is out of the scope of the current work and
it may be a good starting point for future investigations.

\begin{figure}[bt]
    \centering
    \includegraphics[width=0.825\textwidth]{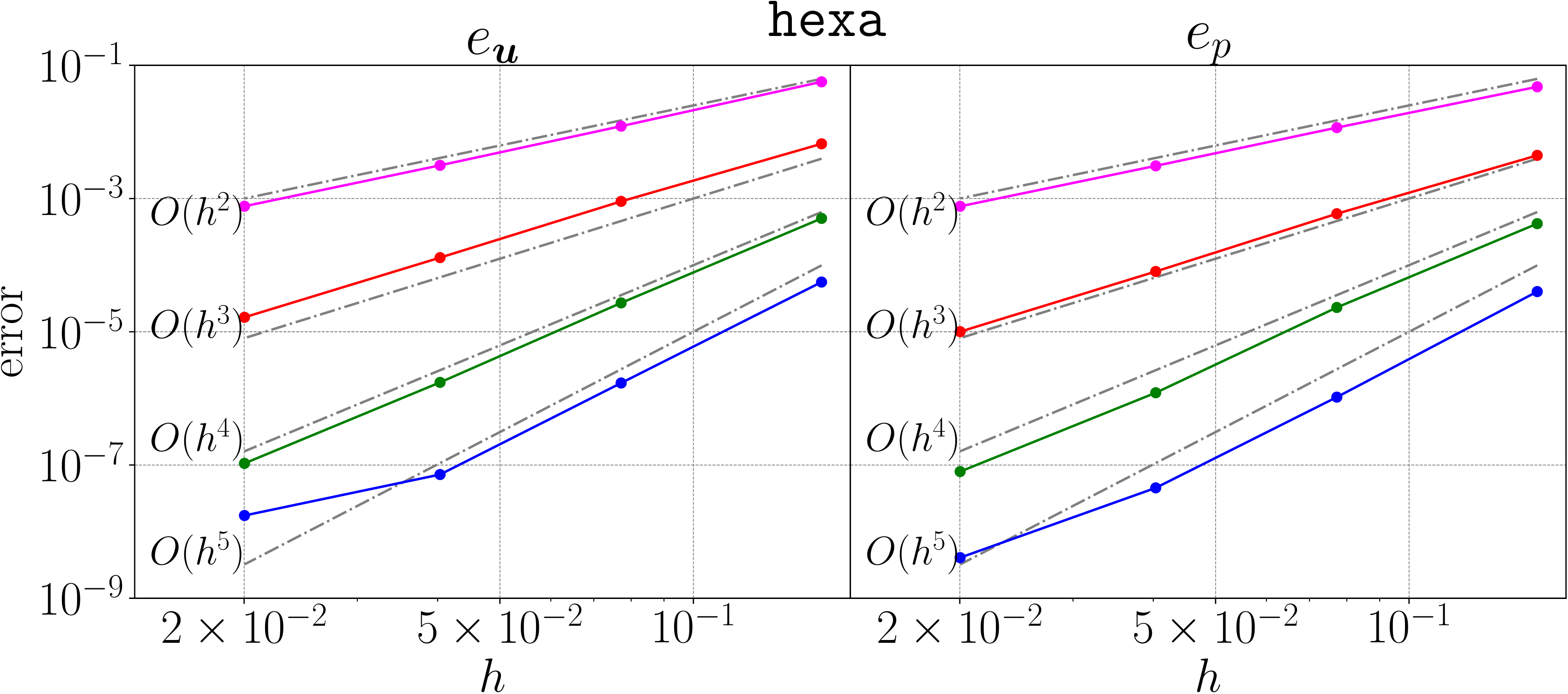}
    \includegraphics[width=0.825\textwidth]{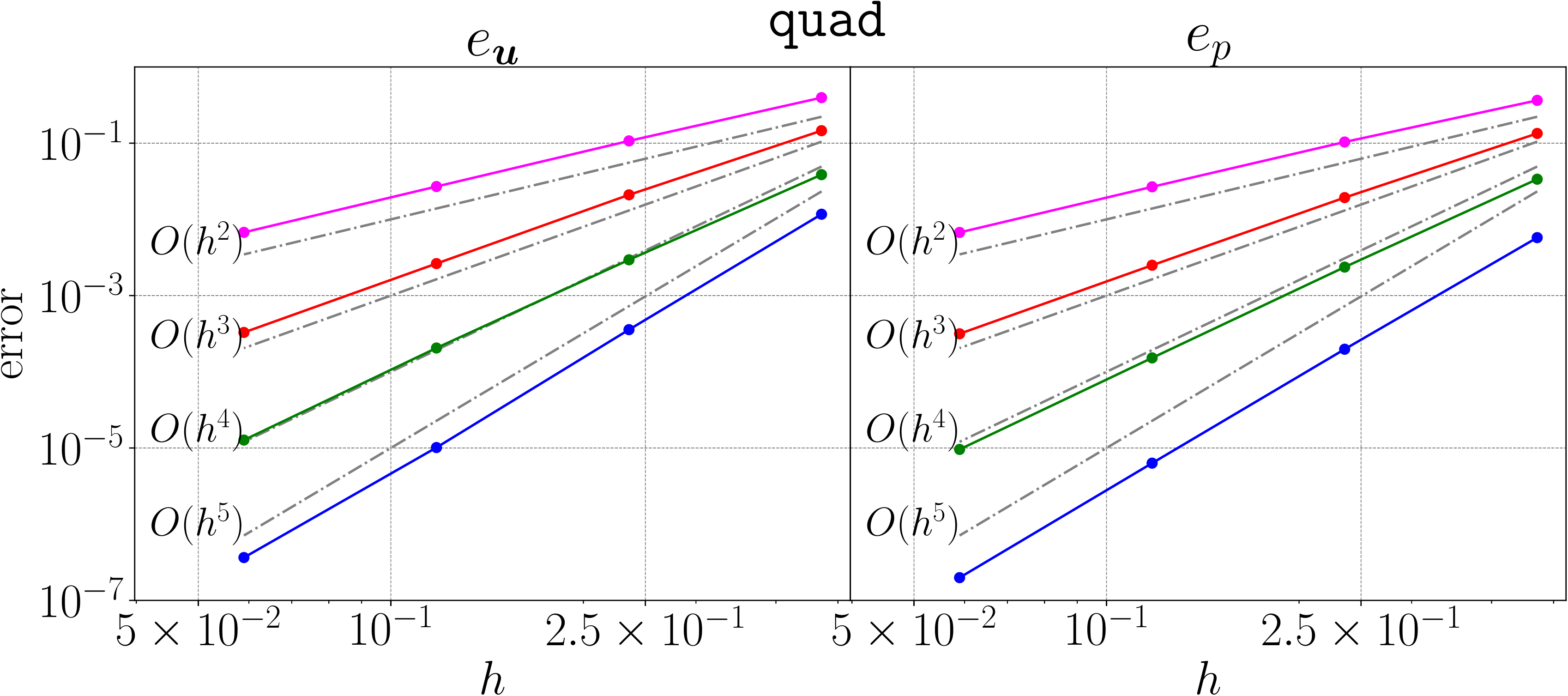}
    \includegraphics[width=0.75\textwidth]{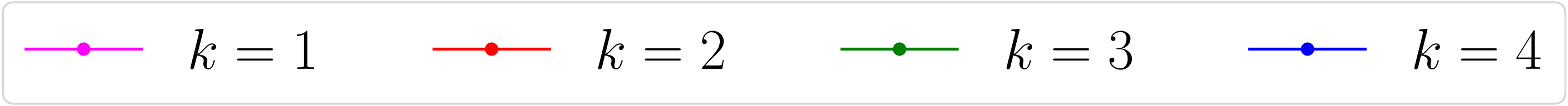}
    \caption{Error decay for different type of meshes. Example in Subsection \ref{subsec:numexe1}.}%
    \label{fig:case1_error_1}
\end{figure}

\begin{figure}[bt]
    \centering
    \includegraphics[width=0.825\textwidth]{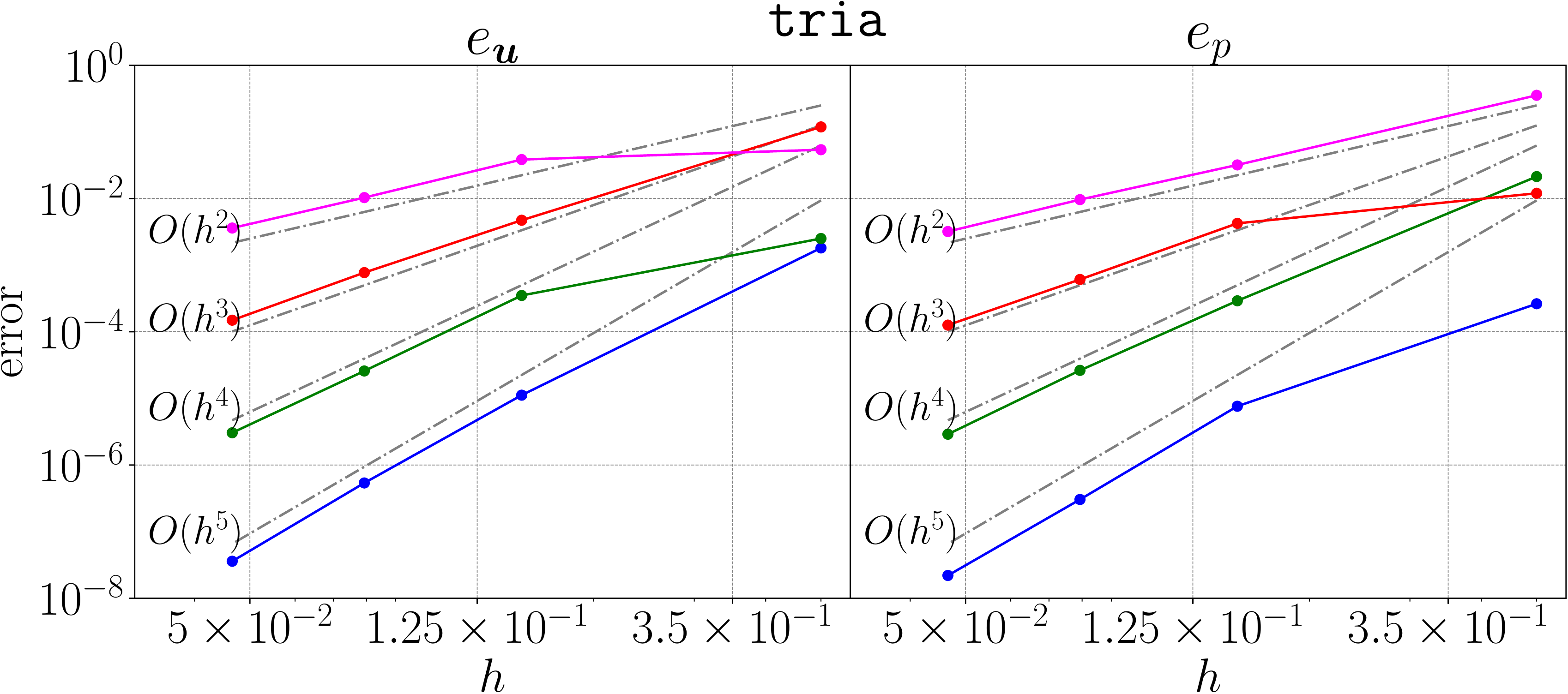}
    \includegraphics[width=0.825\textwidth]{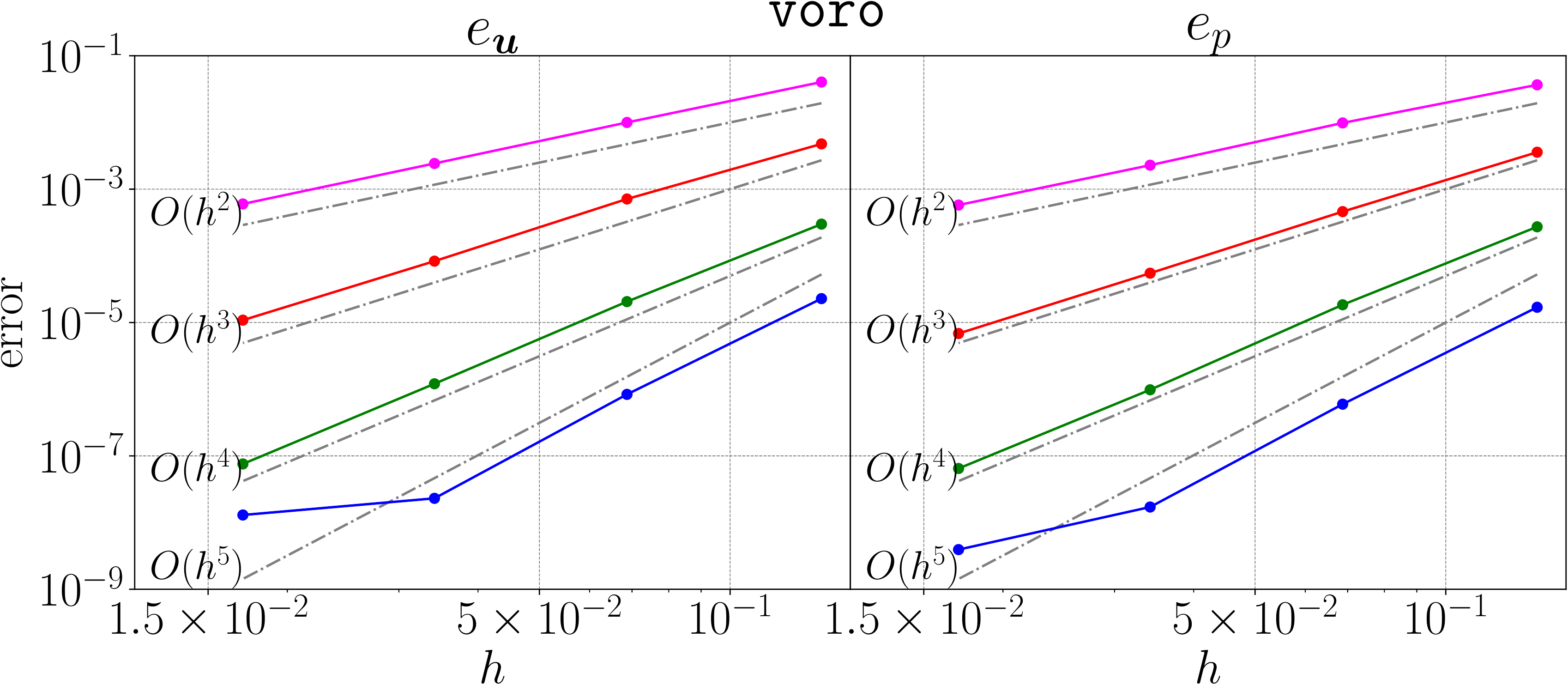}
    \includegraphics[width=0.75\textwidth]{fig/legend}
    \caption{Error decay for different type of meshes. Example in Subsection \ref{subsec:numexe1}.}%
    \label{fig:case1_error_2}
\end{figure}

\subsection{Conservation of energy}\label{subsec:numexe2}

In this second example, we test the conservation of energy given, for the semi-discrete system, in~\eqref{eq:cons-sdisc}.
We consider the same mesh families and approximation degrees of the previous example and
we choose four different time integration schemes:
Explicit Euler, i.e. $\theta = 0$ in \eqref{eq:teta-method_energy}, Implicit Euler, i.e. $\theta = 1$ in \eqref{eq:teta-method_energy} Crank-Nicolson, i.e. $\theta = \frac12$ in \eqref{eq:teta-method_energy}, and Symplectic Euler \eqref{sympEuler}.
 As observed in Section \ref{sub:vem problem}, the space discretization does locally conserve the energy so
the possible lack of conservation will be due to the temporal scheme.

The spatial domain is still the unit square, $\Gamma_D = [0,1] \times\{0,1\}$ i.e. the top and the bottom boundaries, $\Gamma_N = \{0,1\} \times [0,1]$ i.e. the left and right boundaries, and $\Gamma_R=\emptyset$.
The final time is $T=1$ and $\tau=1/200$.
We set to zero the source term, the characteristic velocity $c=1$ and the initial velocity and pressure as
\begin{gather*}
    \bm{u}(\bm{x}, 0) =
    \begin{bmatrix}
       \sin(x)\\
       \cos(y)
    \end{bmatrix},
    \quad \text{and} \quad
    p(\bm{x}, 0) = \cos(2 \pi y) \sin(2 \pi x)\,.
\end{gather*}

The conservation of energy is depicted in Figure~\ref{fig:case2_energy_2}.
In these graphs we compute the difference between the initial energy and the one at a specific time step and
we report such values multiplied by a factor of $10^4$.
We notice that for the Explicit Euler the energy tends to increase during time in agreement with \eqref{eq:Energy_increase}.
This trend becomes more evident for higher approximation degree.
For the Implicit Euler the situation has an opposite behaviour.
Indeed, the energy is now dissipated during time.
Also for this scheme, higher approximation degrees tend to dissipate more energy.
The semi-implicit Symplectic Euler scheme mitigates this effect:
the energy is not conserved during time, but the absolute error does not grow and stay limited.
Finally, Crank-Nicolson shows a perfect energy conservation property. To better appreciate this behaviour we report in Table \ref{tab:energy-error-CN} the values
\begin{equation}\label{eq:energy_error_cn}
    {\rm E} = \frac{\Vert(\bm{u}_{h,\tau}^T, p_{h,\tau}^T)\Vert_{\mathcal{E}_h}
-
\Vert(\bm{u}_{h,\tau}^{0}, p_{h,\tau}^{0})\Vert_{\mathcal{E}_h}}{\Vert(\bm{u}_{h,\tau}^{0}, p_{h,\tau}^{0})\Vert_{\mathcal{E}_h}},
\end{equation}
being $\Vert(\bm{u}_{h,\tau}^T, p_{h,\tau}^T)\Vert_{\mathcal{E}_h}$ the energy of the system at final time $T$.
\begin{table}[h!]
    \centering
    \begin{tabular}{|c|c|c|c|c|}
    \hline
      {\rm E}  & \texttt{tria} & \texttt{quad} & \texttt{hexa} & \texttt{voro}  \\
        \hline
        k = 1 & 7.4038e-16 & 8.6677e-15 & 2.0568e-15 & 9.0382e-15 \\
        k = 2 & 1.2555e-14 & 2.7043e-14 & 2.2665e-15 & 2.4104e-15\\
        k = 3 & 2.8832e-14 & 2.2205e-13 & 6.3020e-16 & 1.0358e-14\\
        k = 4 & 2.9968e-15 & 1.2107e-13 & 2.2041e-16 & 1.6068e-14 \\
        \hline
        \end{tabular}
    \caption{Energy errors {\rm E} computed as in \eqref{eq:energy_error_cn} for the Crank-Nicolson scheme in \eqref{eq:teta-method_energy}. }
    \label{tab:energy-error-CN}
\end{table}
The behaviours of these time discretisation schemes perfectly match the  arguments in Section \ref{sec:time_integration}.

\begin{figure}[!htb]
    \centering
    \includegraphics[width=0.9\textwidth]{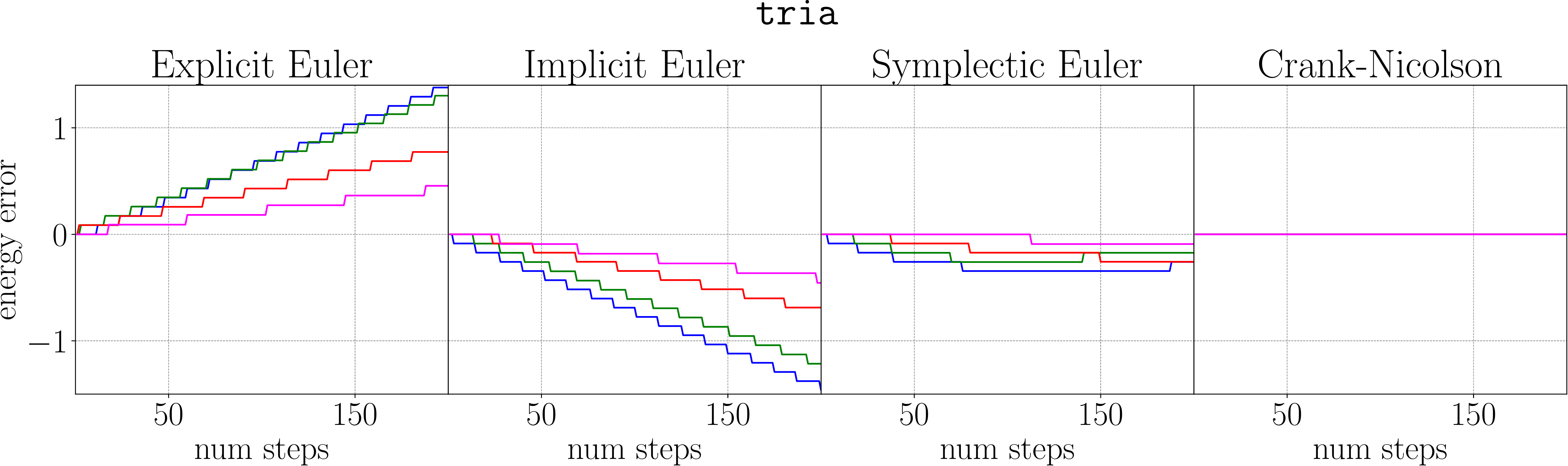}
    \includegraphics[width=0.9\textwidth]{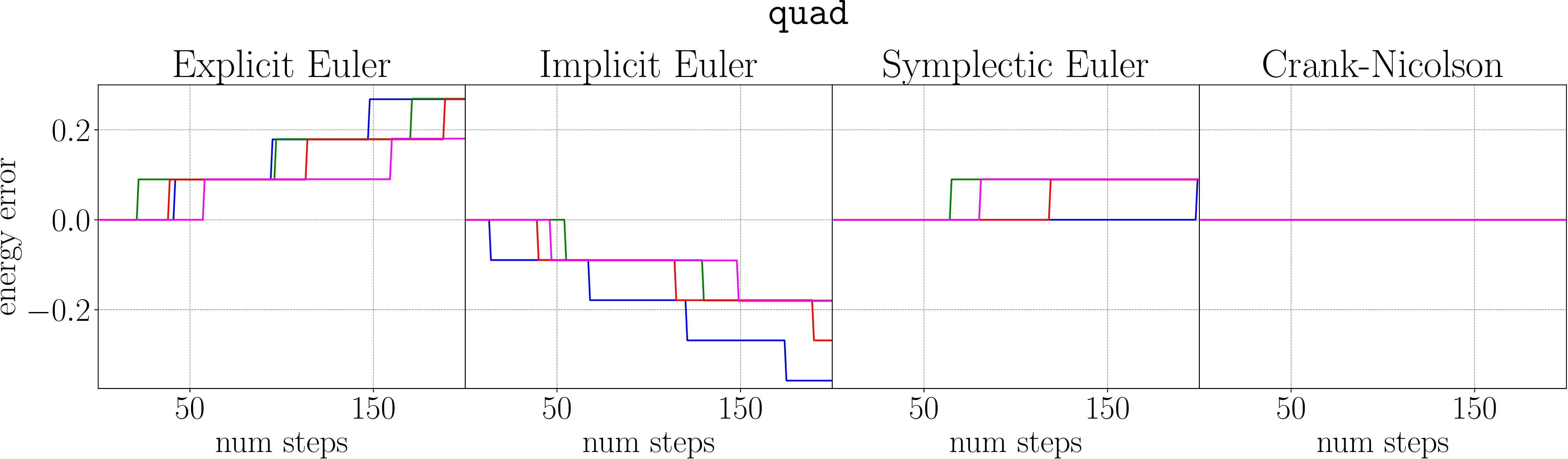}
    \includegraphics[width=0.9\textwidth]{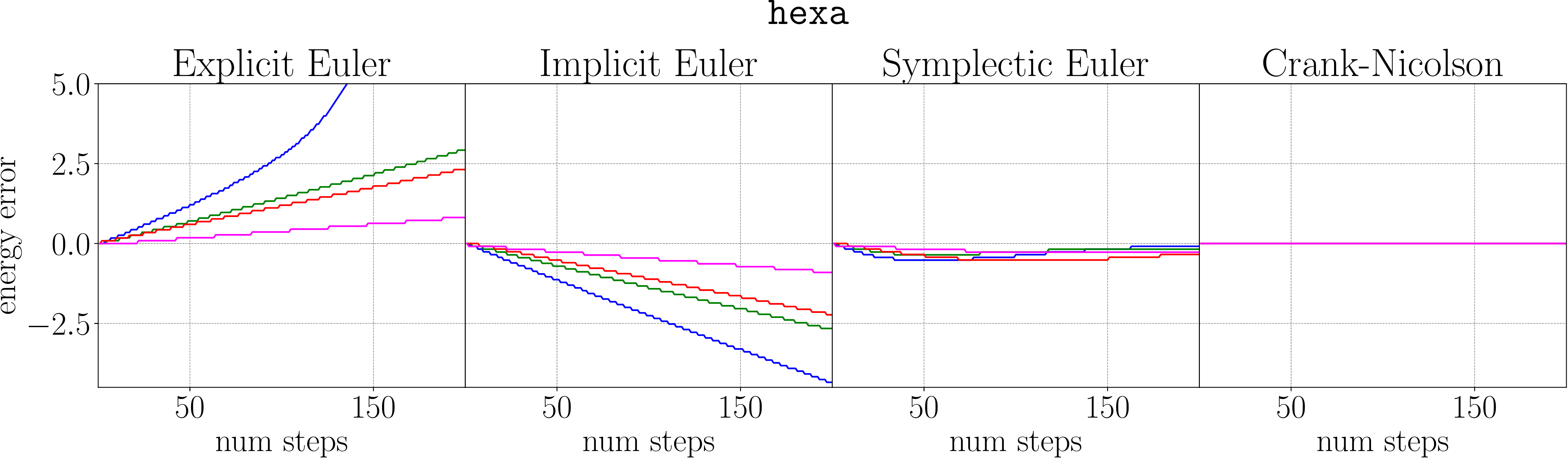}
    \includegraphics[width=0.9\textwidth]{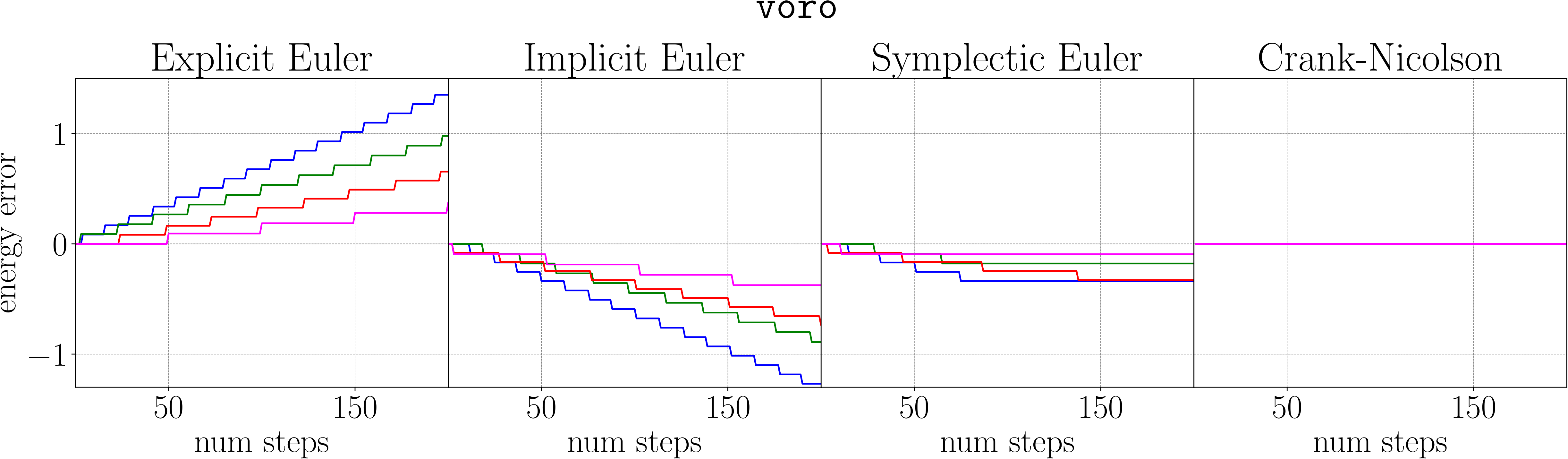}
    \includegraphics[width=0.75\textwidth]{fig/legend}
    \caption{Energy error for different schemes. For visualization purposes the values have been magnified by a factor $10^4$. Example in Subsection~\ref{subsec:numexe2}.}%
    \label{fig:case2_energy_2}
\end{figure}

\subsection{Multiple scattering in curved configurations}\label{subsec:scattering}

In this example, we show the qualitative behaviour of the solution computed by the proposed approach in a domain with multiple circular inclusions.
We consider two different configurations that are inspired by~\cite{barucq2021asymptotic}.
Both domain are defined in $\Omega=(0, 1)^2$, but
the first one has five holes whose centres are located at
\begin{gather*}
    \bm{x}_1 = [a\,,c]^\top, \quad
    \bm{x}_2 = [a\,,b]^\top, \quad
    \bm{x}_3 = [a\,,a]^\top, \quad
    \bm{x}_4 = [b\,,a]^\top, \quad
    \bm{x}_5 = [c\,,a]^\top,
\end{gather*}
while the second one has three additional circles centred at
\begin{gather*}
    \bm{x}_6 = [b\,,c]^\top, \quad
    \bm{x}_7 = [c\,,c]^\top, \quad
    \bm{x}_8 = [c\,,b]^\top,
\end{gather*}
where the values of the three parameters are given by
\begin{gather*}
    a = 0.6052631578947355\,,\quad
    b =0.5,
    \quad\text{and}\quad
    c=0.3947368421052622\,.
\end{gather*}
All circles have diameter 0.02. We refer to the first and the second case as \texttt{fiveHoles} and \texttt{eightHoles}, respectively.
The grid is build starting from a structured quadrilateral mesh,
where the cells intersecting with a circle are properly cut.
Moreover, since we exploit the possibility to include the geometry information within VEM spaces~\cite{dassi2021},
we do not need to over-refine edges to get an accurate representation of the circles themselves.
As a consequence, we can reduce the computational effort without losing the accuracy of the solution,
by neglecting spurious waves due to a piece-wise linear representation of the circles.

For both \texttt{fiveHoles} and \texttt{eightHoles} examples,
we build two meshes, see Figure~\ref{fig:holes}.
The first one starting from a square divided in $38\times 38$ uniform squares.
Then, the second one is constructed by refining 2 times only the quadrilateral elements of the first one,
see the detail in Figure~\ref{fig:holesDet}.

\begin{figure}[!htb]
    \centering
    \begin{tabular}{cc}
    \includegraphics[width=0.40\textwidth]{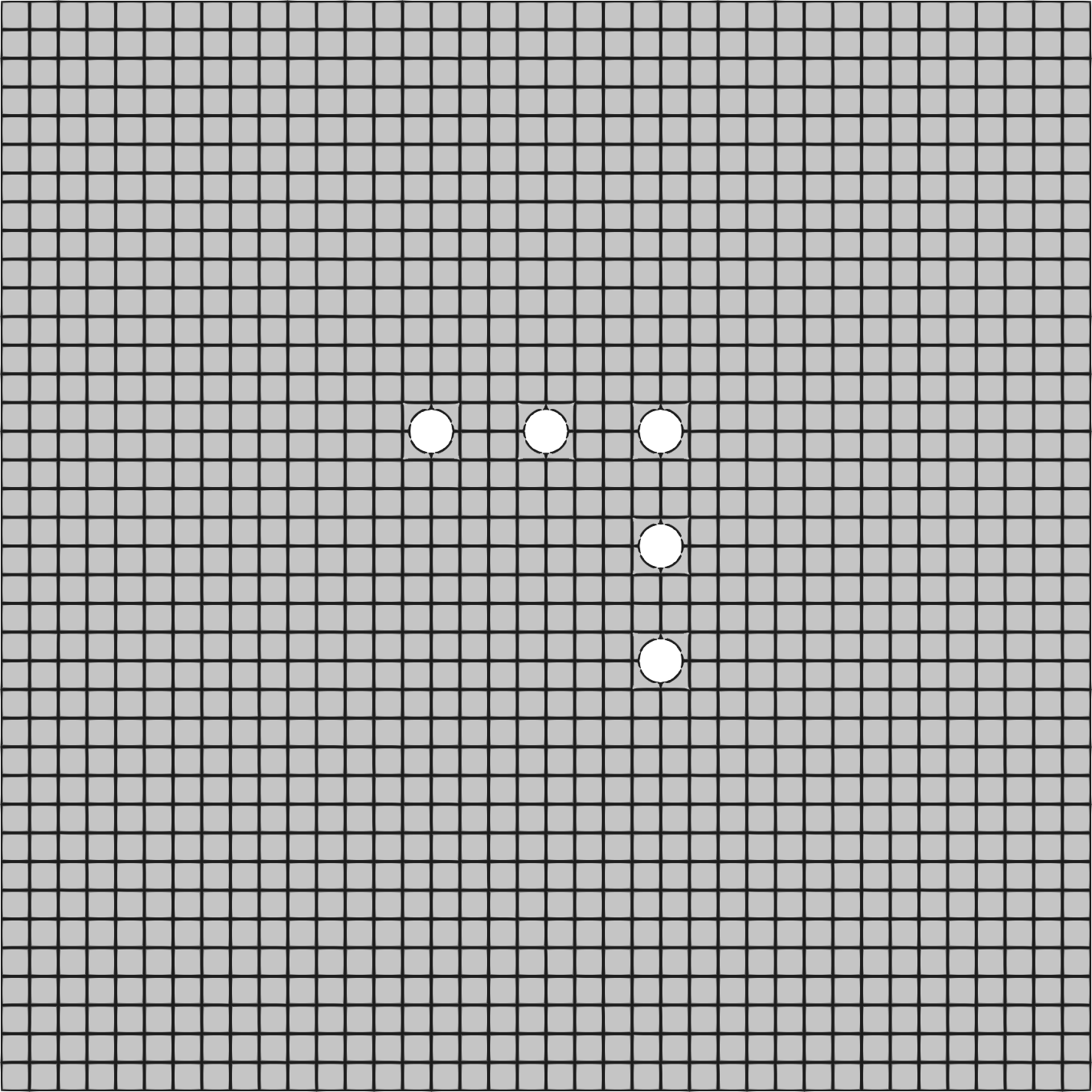}&
    \includegraphics[width=0.40\textwidth]{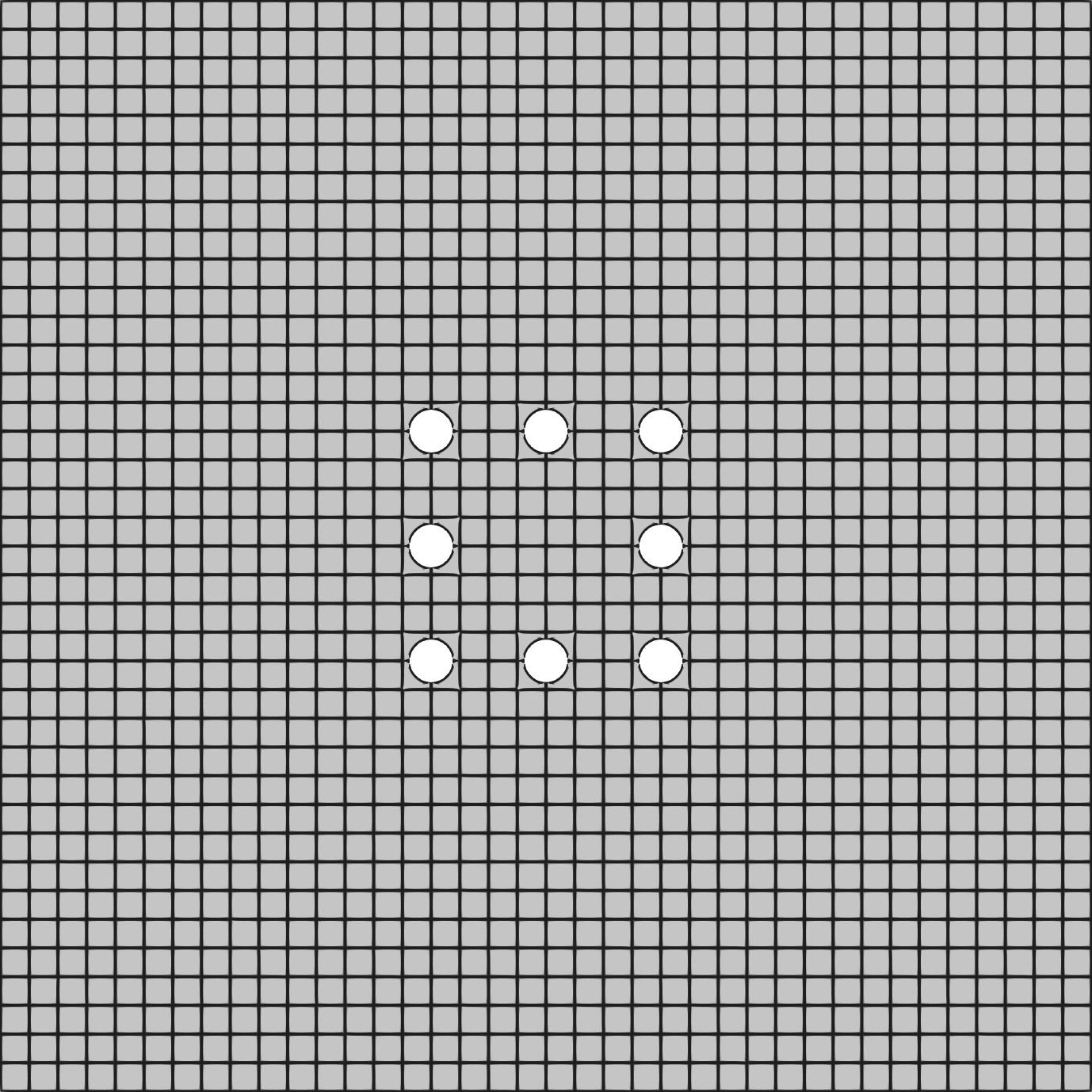}
    \end{tabular}
    \caption{Representation of the computational grid for \texttt{fiveHoles} (left)
    for \texttt{eightHoles} (right). Example in Subsection~\ref{subsec:scattering}.}
    \label{fig:holes}
\end{figure}

\begin{figure}[!htb]
    \centering
    \begin{tabular}{cc}
    \includegraphics[width=0.40\textwidth]{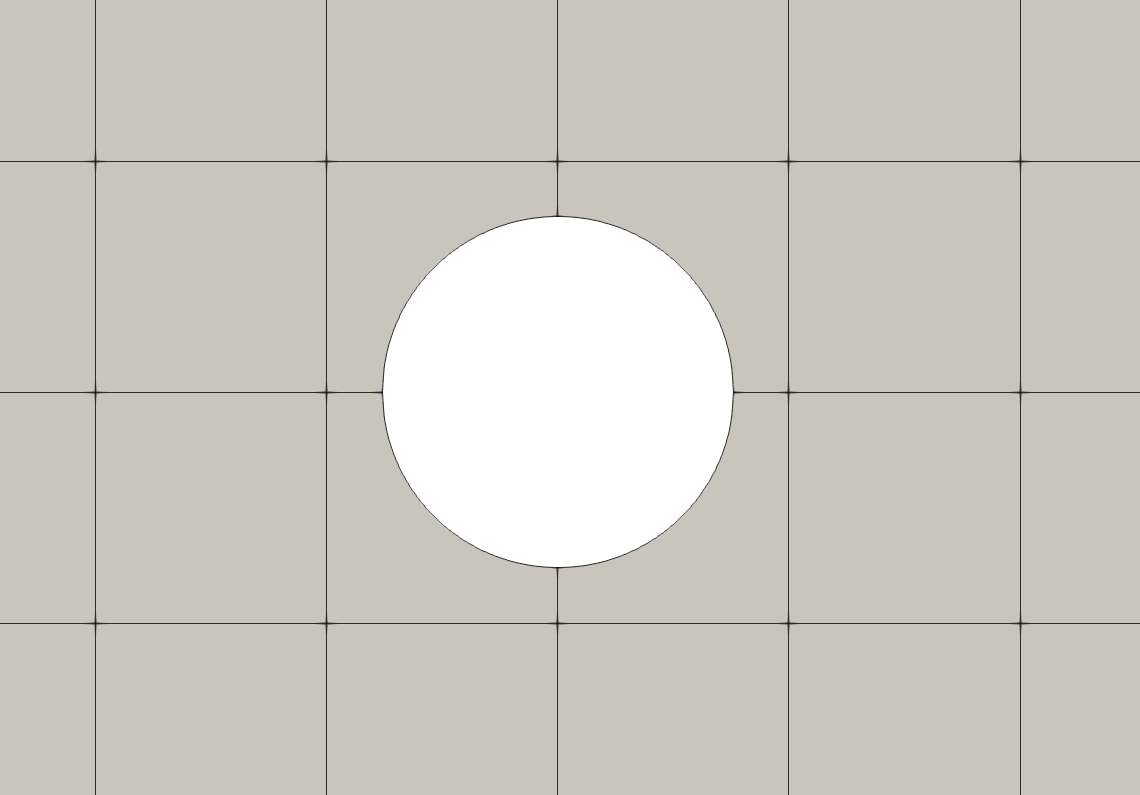}&
    \includegraphics[width=0.40\textwidth]{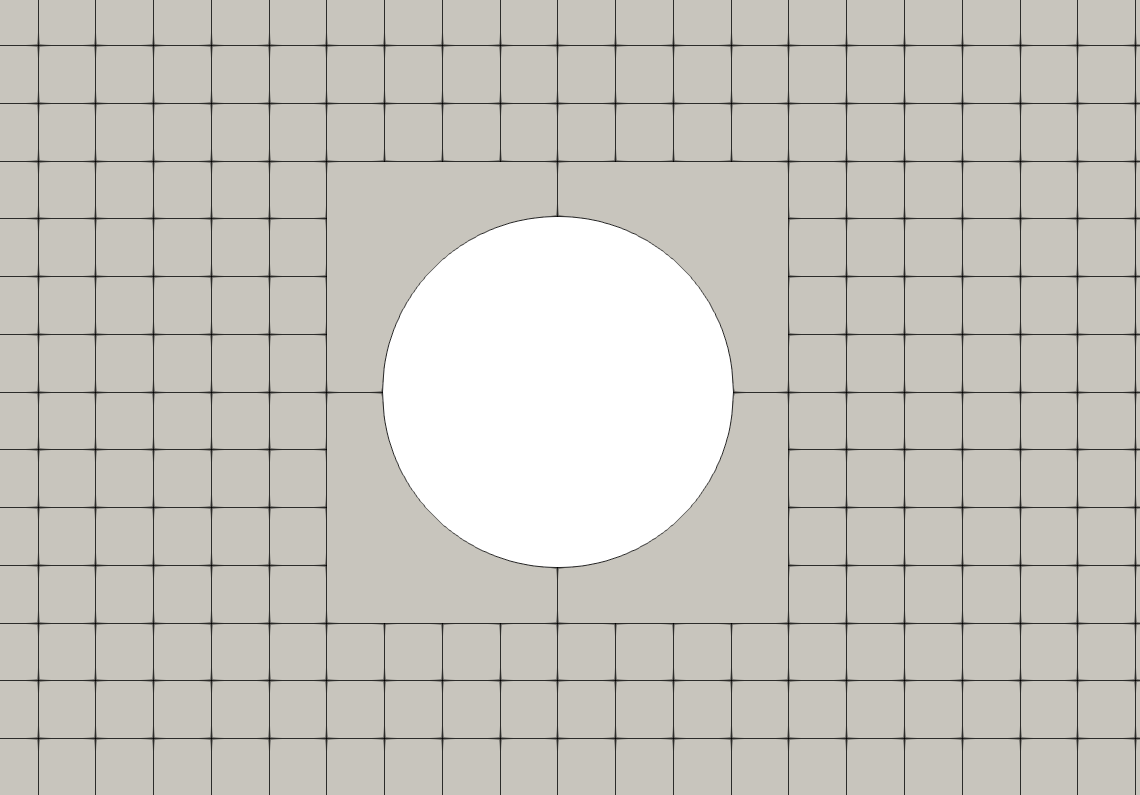}
    \end{tabular}
    \caption{Detail of one hole in the inital mesh (left) and in the refined one (right). Example in Subsection~\ref{subsec:scattering}.}
    \label{fig:holesDet}
\end{figure}

Along the external boundaries of $\Omega$, we impose absorbing boundary conditions, while on
all the small circles perimeters we use homogeneous Neumann conditions.
Then, we consider $c=1$ and the following source term
\begin{gather*}
    f(\bm{x}, t) = 10.\,e^{-(t-1)^2/\sigma_1} e^{-\norm{\bm{x}-\bm{x}_c}^2/\sigma_2},
\end{gather*}
with $\sigma_1 = 0.01$ and $\sigma_2 = 0.00125$ and $\bm{x}_c=[0.5, 0.5]^\top$.
The initial values of both $\bm{u}$ and $p$ are set to zero.
We assume approximation degree equal to $k=2$.
As final time we take $T=10$, we divide the interval $[0,T]$ into 50 equally spaced sub-intervals and we use the Crank-Nicolson time integration scheme.

In Figure~\ref{exe5Diff} we show the computed pressure solution $p_h$ evaluated at quadrature points at different time steps for the \texttt{fiveHoles} case.
On the top panels we show the coarse mesh while,
on the bottom ones, we consider the fine mesh.
From a qualitative point of view we observe no much difference between them.
\begin{figure}[tbp]
    \centering
    \setlength{\tabcolsep}{0.01\textwidth}
    \begin{tabular}{ccc}
        step 13 & step 26 & step 42\\
        \includegraphics[width=0.31\textwidth]{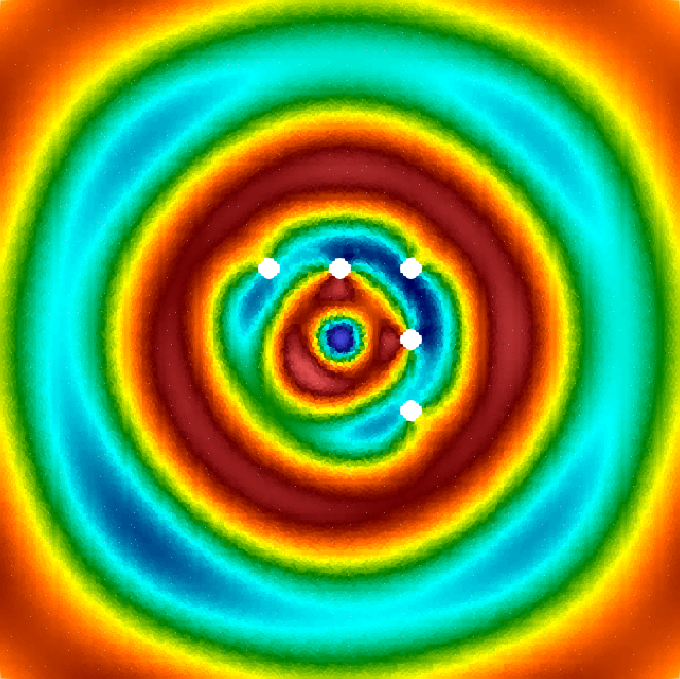}&
        \includegraphics[width=0.31\textwidth]{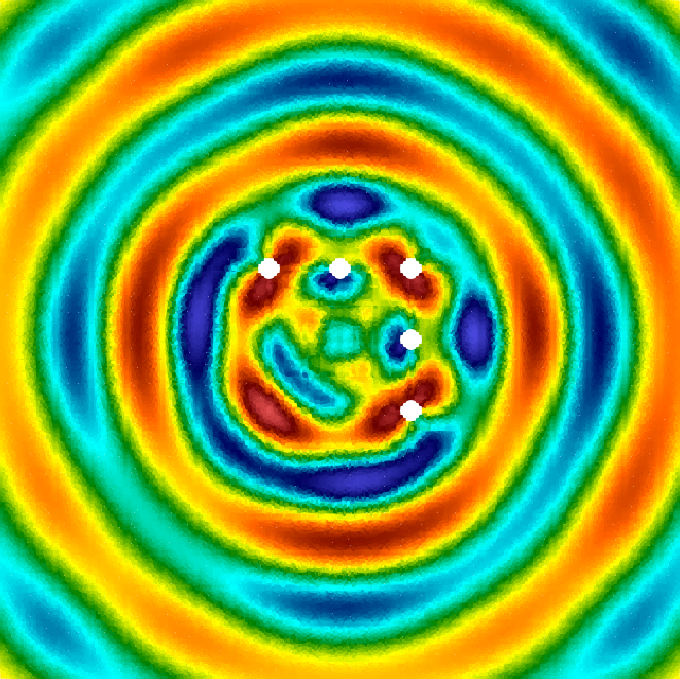}&
        \includegraphics[width=0.31\textwidth]{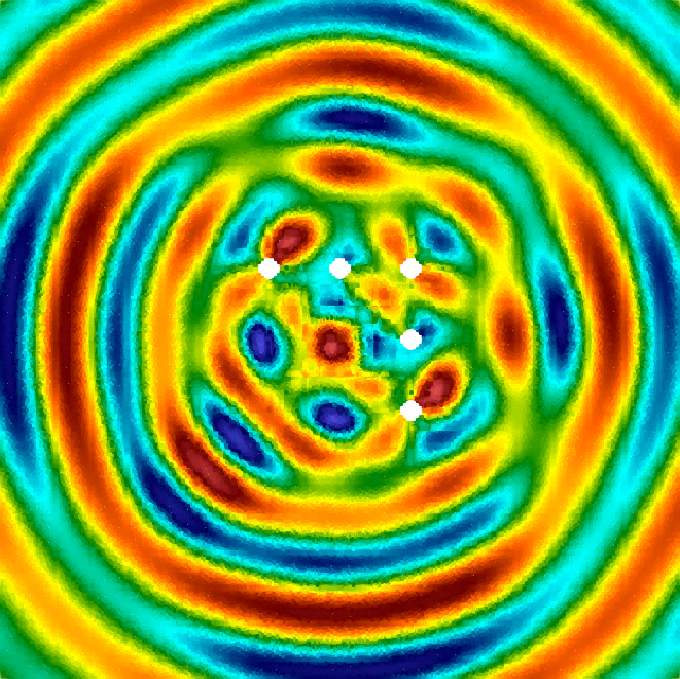}\\[0.01\textwidth]
        \includegraphics[width=0.31\textwidth]{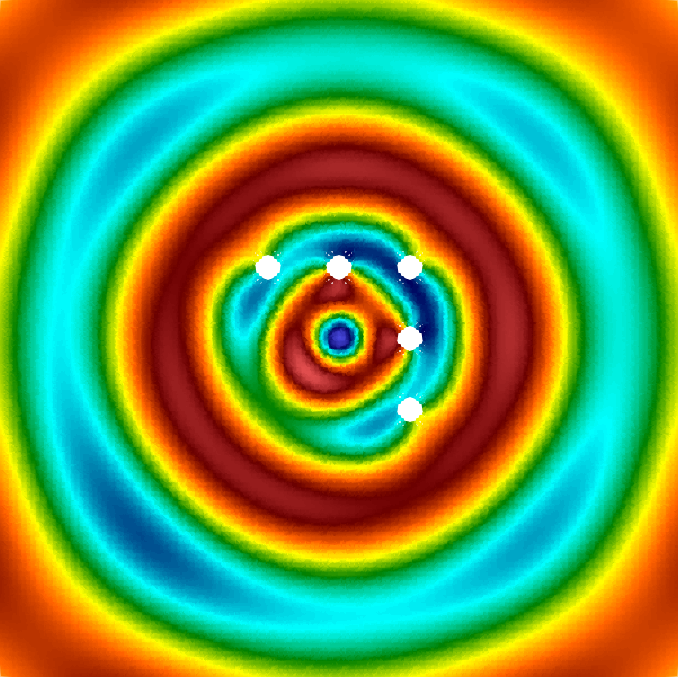}&
        \includegraphics[width=0.31\textwidth]{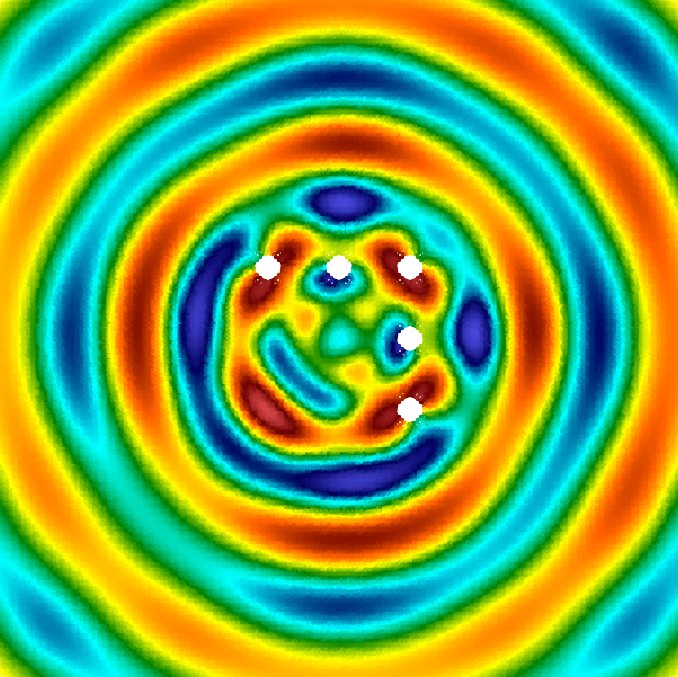}&
        \includegraphics[width=0.31\textwidth]{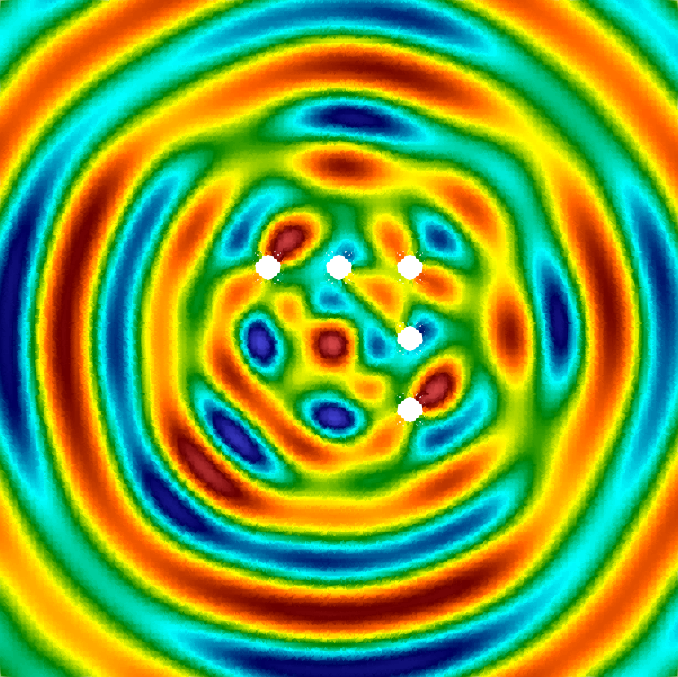}
    \end{tabular}
    \caption{Solution obtained with the \texttt{fiveHoles} mesh, on the top with a coarse grid and on the bottom with a finer grid. Example in Subsection~\ref{subsec:scattering}.}
    \label{exe5Diff}
\end{figure}

In Figure~\ref{exe8Diff} we consider the \texttt{eightHoles} case.
Also for this experiment we observe that the coarser mesh provides similar results of the ones obtained with the refined mesh.
Moreover, since the holes are displaced in a symmetric way with respect to $\bm x_c$
we expect the waves preserve a symmetric structure.
This is confirmed by the snapshots represented in Figure~\ref{exe8Diff}.
\begin{figure}[tbp]
    \centering
    \setlength{\tabcolsep}{0.01\textwidth}
    \begin{tabular}{ccc}
    step 13 & step 26 & step 42\\
    \includegraphics[width=0.31\textwidth]{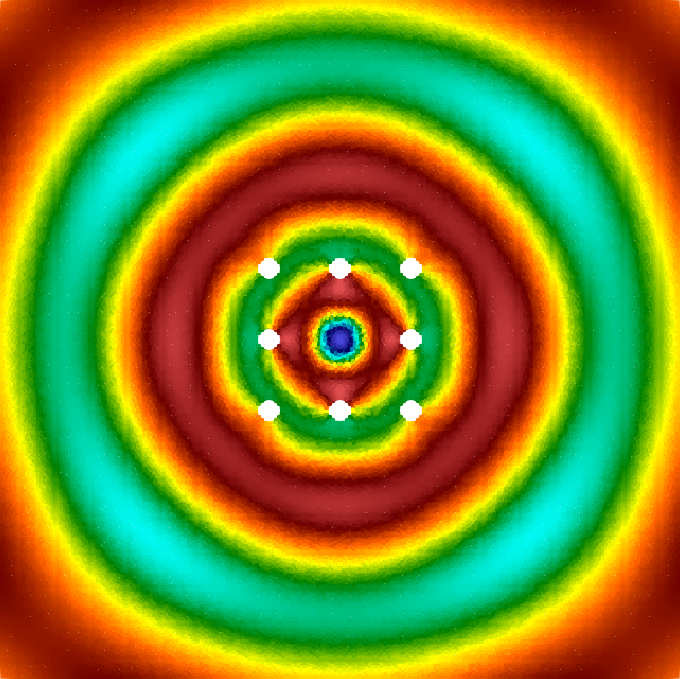}&
    \includegraphics[width=0.31\textwidth]{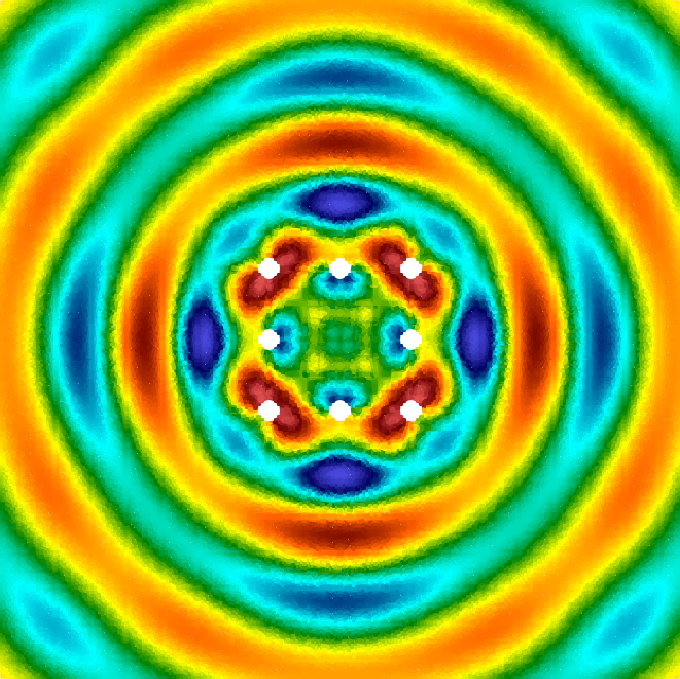}&
    \includegraphics[width=0.31\textwidth]{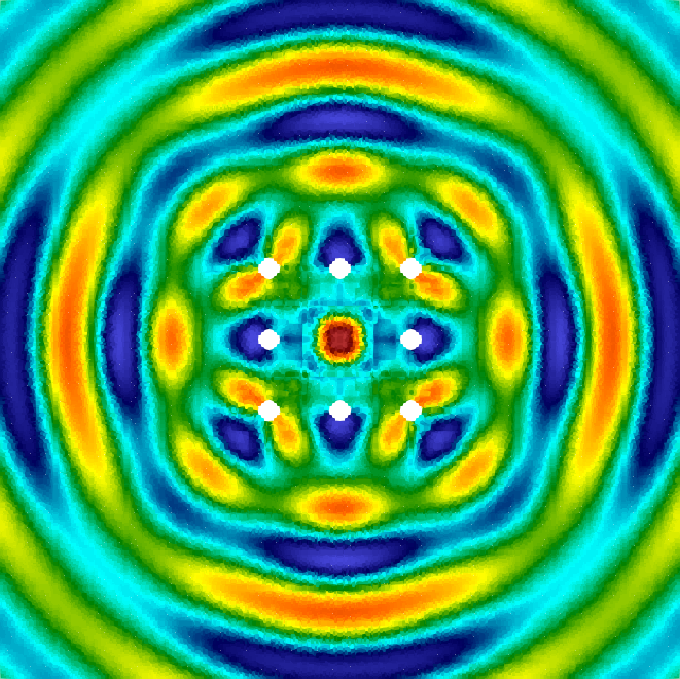}\\[0.01\textwidth]
    \includegraphics[width=0.31\textwidth]{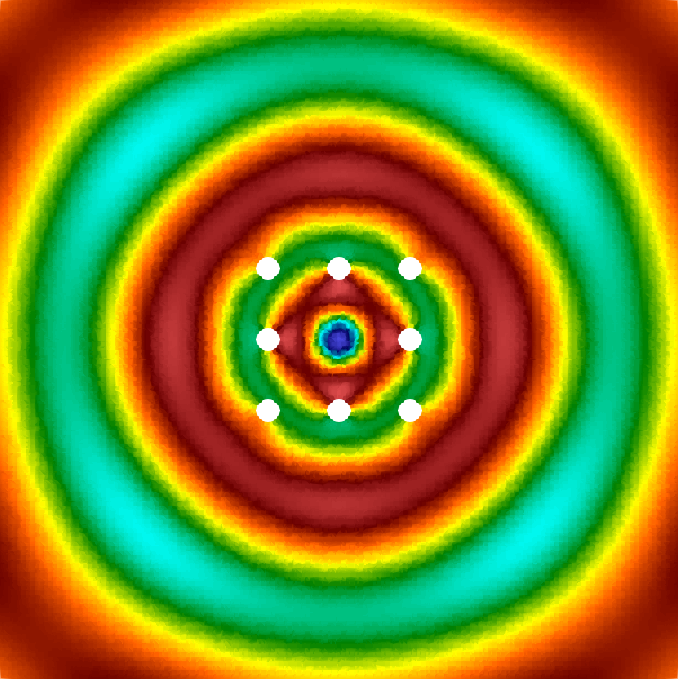}&
    \includegraphics[width=0.31\textwidth]{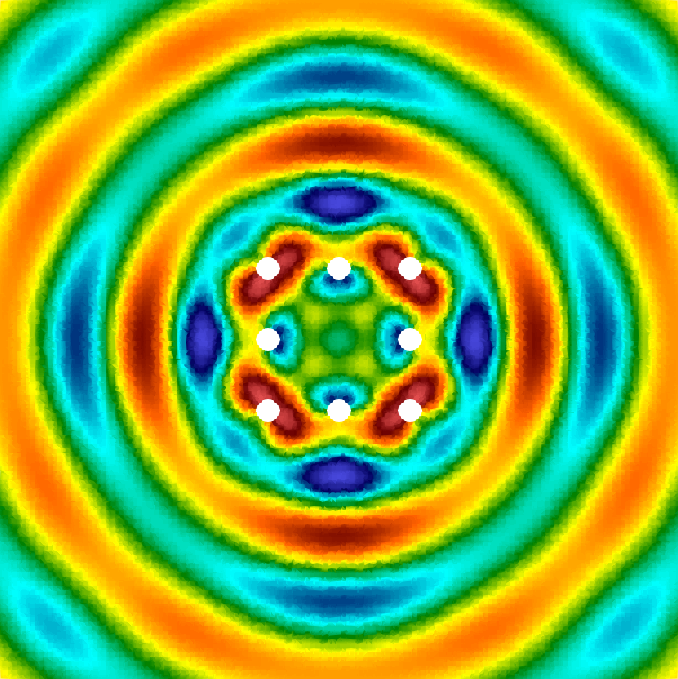}&
    \includegraphics[width=0.31\textwidth]{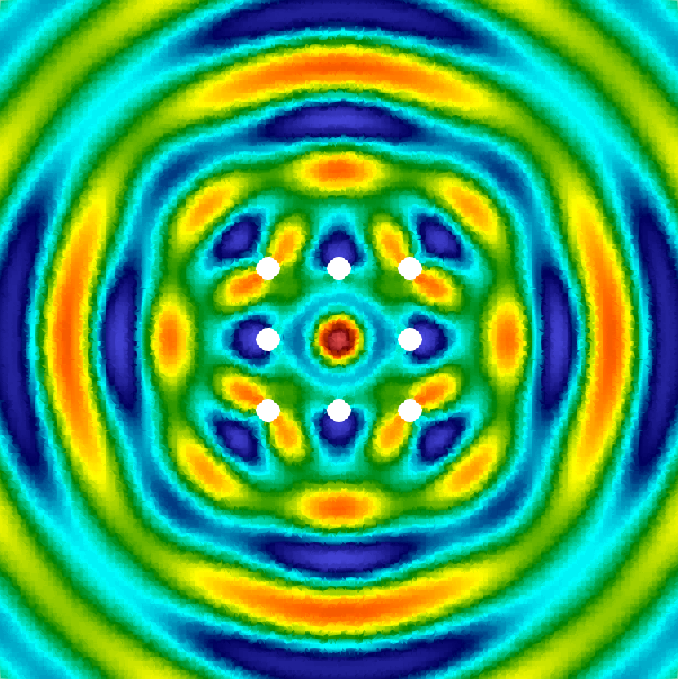}
    \end{tabular}
    \caption{Solution obtained with the \texttt{eightHoles} mesh, on the top with a coarse grid and on the bottom with a finer grid. Example in Subsection~\ref{subsec:scattering}.}
    \label{exe8Diff}
\end{figure}

\section{Conclusions}\label{sec:conclusion}

In this work, we presented a general order virtual element numerical discretization scheme to approximate the wave equation written in mixed form. The latter results in a first order system of differential equations in both space and time dimension.
We presented the \textit{a-priori} stability analysis as well as the convergence property of the scheme in a suitable energy norm and we showed that the proposed virtual element scheme preserve the semi-discrete energy of the system in absence of external load. To integrate in time the semi-discrete problem we consider the family of the $\theta$ - method schemes and we discuss their energy conservation properties.
We verified the theoretical results on different benchmark tests and we applied the proposed scheme on a domain with circular inclusions to show the capabilities of the method in term of accuracy and of flexibility in handling complex geometries.
To conclude, the presented mixed virtual element method allows a robust and flexible numerical discretization that can be successfully applied to wave propagation problems. Future developments in this direction may include the study of multi-physics problems (written in a mixed form) such as vibro-acoustics (with elastic or poroelastic structure) interaction problems.

\section*{Acknowledgments}
The authors acknowledge the financial support of INdAM-GNCS through Project ``Sviluppo ed analisi di Metodi agli Elementi Virtuali per
processi accoppiati su geometrie complesse''. Moreover, the authors would like to thank Anna Scotti for many fruitful discussions.


\bibliographystyle{plain}






\end{document}